\documentclass[12pt]{amsart}
\pdfoutput=1

\usepackage{amssymb, amsmath}
\usepackage{enumerate}
\usepackage{graphicx}
\usepackage{amsthm}

\makeatletter
\@namedef{subjclassname@2010}{%
  \textup{2010} Mathematics Subject Classification}
\makeatother

\newtheorem{thm}{Theorem}[section]
\newtheorem{cor}[thm]{Corollary}
\newtheorem{lem}[thm]{Lemma}

\theoremstyle{definition}
\newtheorem{defin}[thm]{Definition}
\newtheorem{rem}[thm]{Remark}

\numberwithin{equation}{section}

\def\raw{\rightarrow}

\def\R{{\mathbb R}}
\def\Z{{\mathbb Z}}
\def\N{{\mathbb N}}

\def\C{{\mathbb C}}

\def\vv{{\bf v}}

\def\raw{\rightarrow}

\def\Inv{^{-1}}

\def\ty{{\tilde y}}
\def\tI{{\tilde I}}

\def\tUpsilon{\tilde{\Upsilon}}

\def\baru{{\overline u}}
\def\barv{{\overline v}}
\def\bara{{\overline a}}
\def\bargamma{{\overline \gamma}}
\def\barGamma{{\overline \Gamma}}
\def\bell{{\overline \ell}}
\def\barm{{\overline m}}
\def\bark{{\overline k}}

\def\cN{\mathcal{N}}

\def\cA{{\cal A}}

\DeclareMathOperator{\LP}{LimPts}

\DeclareMathOperator{\sppan}{span}
\DeclareMathOperator{\image}{image}

\def\vv{\vec{v}}
\def\vu{\vec{u}}
\def\vw{\vec{w}}

\def\vw{\vec{w}}
\def\ve{\vec{e}}
\def\vell{\vec{\ell}}

\def\homeo{homeomorphism}

\def\PF{Perron–-Frobenius}

\def\be{\mathbf{e}}

\def\cA{\mathcal{A}}
\def\cC{\mathcal{C}}
\def\cS{\mathcal{S}}

\def\cW{\mathcal{W}}
\def\cP{\mathcal{P}}
\def\cR{\mathcal{R}}

\def\hy{\hat{y}}
\def\hone{{\hat{1}}}
\def\hbeta{{\hat{\beta}}}

\def\us{{\underline s}} 
\def\ut{{\underline t}} 
\def\uv{{\underline v}} 
\def\un{{\underline n}} 
\def\uw{{\underline w}} 

\def\un{{\underline n}}

\def\Cl{\operatorname{Cl}}

\def\spec{\cS}
\def\sspec{\cS\cS}

\def\Gammam{\Gamma_{(m)}}

\def\uht{\underline{\hat{t}}}
\def\uhs{\underline{\hat{s}}}
\def\uhw{\underline{\hat{w}}}
\def\Dun{{\Delta_{\un}}}
\def\Dunp{{\Delta_{\un}^+}}

\def\tOmega{\tilde{\Omega}}
\def\sne{\not=\not=}

\def\barF{\overline{F}}
\def\barf{\overline{f}}

\def\mot{{\mu_1, \mu_2}}
\def\Tmot{T_\mot}
\def\tTmot{\tilde{T}_\mot}
\def\tT{\tilde{T}}

\newcommand{\lseq}[1]{\underleftarrow{#1}}
\newcommand{\rseq}[1]{\underrightarrow{#1}}

\def\Np{{\N^+}}
\def\what{{\Sigma^+_\infty}}

\AtBeginDocument{%
   \def\MR#1{}
}

\begin{document}


\baselineskip=17pt



\title[Infimax family geometric representation]{Geometric representation of the infimax S-adic family}

\author[P. Boyland]{Philip Boyland}
\address{Department of Mathematics\\
University of Florida\\
Gainesville, FL 32611}
\email{boyland@ufl.edu}

\author[W. Severa]{William Severa}
\address{Department of Mathematics\\
University of Florida\\
Gainesville, FL 32611}
\email{wmsever@sandia.gov}

\date{}

\begin{abstract}
We construct geometric realizations
for the infimax family of substitutions by
generalizing the
Rauzy-Canterini-Siegel method for a single
substitution to the S-adic case.
The composition of each countably infinite subcollection
of substitutions from the family has an asymptotic fixed sequence
whose shift orbit closure is an infimax minimal
set $\Delta^+$.
The subcollection of substitutions also generates
an infinite Bratteli-Vershik diagram with prefix-suffix labeled 
edges. Paths in the diagram give the Dumont-Thomas expansion
of  sequences in $\Delta^+$ which in turn gives
a projection onto the asymptotic stable direction
of the infinite product of the Abelianization matrices.
The projections of all sequences from $\Delta^+$ is
the generalized Rauzy fractal which
has subpieces corresponding to the images of
symbolic cylinder sets.
The intervals containing these subpieces are shown
to be disjoint except at endpoints, and thus the
induced map derived from the symbolic shift translates them.
Therefore the process yields an Interval Translation Map (ITM), and
the Rauzy fractal is proved to be its attractor.
\end{abstract}



\maketitle

\section{Introduction}
Substitution morphisms are an integral part of many areas of
mathematics including 
dynamical systems, combinatorics, number theory and
formal language theory. The books  \cite{Q}, \cite{Fogg}, and \cite{ANT} 
give a good sense of the diversity and depth of the field.
The natural
generalization of a single substitution 
to the composition of an infinite sequence of substitutions, 
termed S-adic systems by Ferenczi in \cite{F1},
allows the modeling and analysis of a wider range of fundamental
structures (see, for example, \cite{Sadic}, \cite{Sadicthesis},
\cite{Sadic2}, and \cite{BST}).

In \cite{lexi} an S-adic family was found to generate the solutions to 
the following problem. A one-sided sequence is called
\textit{maximal} if it is larger in the lexicographic order
 than all its shifts.
Let ${\mathcal M}(\rho)$ denote all the maximal sequences
with asymptotic digit frequency vector $\rho$. The infimum 
 in the lexicographic order
of ${\mathcal M}(\rho)$ is called the \textit{infimax sequence for}
$\rho$ and in \cite{lexi} it was shown that it
can be constructed using a specific S-adic family. In this paper
we further study the properties of this 
\textit{infimax S-adic family}.
The family studied here is indexed by 
the positive integers $\Np = {1, 2, 3, \dots}$ with
the substitution $\Lambda_n$ given by
\begin{equation}\label{subdef}
\Lambda_n: 
\begin{array}{ccl} 1 & \rightarrow & 2 \\  
2 & \rightarrow & 3 1^{n +1}\\ 
3 & \rightarrow & 31^n.
\end{array} 
\end{equation}
Note that the substitution $\Lambda_0$ is required
for the full infimax problem: see Remark~\ref{missing}(b) below.

The collection of allowable lists of indices is
$\what := (\Np)^{\Np}$. 
Given a list of indices $\un \in \what$, it is easy to
check that the right one-sided sequence 
\begin{equation}\label{alphaone}
\rseq{\alpha} := \lim_{k\raw\infty}
 \Lambda_{n_1}\circ\Lambda_{n_2}\dots\Lambda_{n_k}(3)
\end{equation}
exists. 
For a constant list of indices $\un = n n n \dots$,
the sequence
$\rseq{\alpha}$ is a fixed point of $\Lambda_n$. For a general $\un\in\what$,
the sequence
$\rseq{\alpha}$ is an asymptotic fixed point in the sense that for any
one-sided sequence $\rseq{s} = s_0 s_1 \dots$ with $s_i \in \{1,2,3\}$,
\begin{equation*}
 \lim_{k\raw\infty} 
\Lambda_{n_1} \Lambda_{n_2} \dots \Lambda_{n_k}(\rseq{s}) = \rseq{\alpha}.
\end{equation*}
Returning to the infimax problem,
for a given asymptotic digit frequency vector $\rho$, 
it was shown in \cite{lexi} that  a three dimensional
continued fraction algorithm using $\rho$ as an
input generates a list $\un \in \what$. The 
infimax for $\rho$ is then  the corresponding $\rseq{\alpha}$
as defined in \eqref{alphaone}.

The main dynamical object of study is  
the orbit closure of $\rseq{\alpha}$ under the left shift,
or $\Dunp := \Cl(o^+(\rseq{\alpha}, S))$. Since
$\Dunp$ is always dynamically minimal, we call it an
\textit{infimax minimal set}.
The study of $\Dunp$ and its associated substitutions
is facilitated
by finding a geometric representation of $\Dunp$ as defined
in \cite{Q}. This
means a concrete, geometrically defined model system
$T:X\raw X$ so that the dynamics of $\Dunp$ under
the shift are embedded in the dynamics of $T$ on $X$ in
a nice way. Specifically, we seek a map $\Upsilon^+:\Dunp\raw
X$ which is a conjugacy,
 $ \Upsilon^+\circ S = T \circ \Upsilon^+$ on the
image of $\Upsilon^+$. In addition, the space $X$ is required to
have a nice partition by which $T$-orbits are coded
so that corresponding sequences are recovered as itineraries.
Geometric representations have been found for many substitutions 
(see \cite{S1} for a summary). 
Depending on the circumstances, 
 different requirements can be made on the map $\Upsilon^+$.
From the ergodic theory perspective, one would like $\Upsilon^+$ to
be  measure preserving and injective 
almost everywhere with respect to the appropriate
invariant measures. We work here in the topological
category and so $\Upsilon^+$ is required to be continuous
and, in fact, $\Upsilon^+$ will be a \homeo\ onto its image.

The geometric models for the S-adic infimax family given here
are elements of a two-parameter 
family of interval translation maps (ITM). These 
maps are generalizations of interval exchange transformations
in which the images of intervals are allowed to overlap (\cite{BK}).
The \textit{attractor} of an ITM is the intersection of
the forward iterates of the entire interval.
For each $\un\in\what$, we show in Theorem~\ref{maintext} that the symbolic
infimax minimal set $\Dunp$ is conjugate via a map $\Upsilon^+$
to the attractor of a (slightly extended) ITM.
The extension of the ITM is necessary  
to make it continuous as is often done with interval
exchange maps (\cite{Keane}).
\begin{thm}\label{main} For each $\un\in\what$, 
the three symbol infimax minimal set has a geometric representation
as the attractor of an Interval Exchange Map (ITM) on three intervals.
\end{thm}
In addition, the conjugacy respects
order structures: the lexicographic order on the symbolic
minimal set is reversed under the conjugacy onto the
ITM attractor inside the unit interval.

Previously Bruin and Troubetzkoy
have shown that each $\Dun$ is isomorphic to the attractor of
an ITM (\cite{BT},
see also Section 5 in \cite{BK}). We
extend these results obtaining a full homeomorphic conjugacy
and in addition study the two-sided infimax minimal
set. The starting point in \cite{BT} was the natural
 renormalizations of a family of ITM. The substitutions
\eqref{subdef} then arise as the symbolic descriptors of this process. 
On the other hand, motivated by the infimax problem,
we begin here with the substitutions and sequences themselves.
 Using  generalizations of methods commonly used for single
substitutions we find  
the ITM geometric representation as a direct and natural consequence
of the structure of the generalized Rauzy
fractal and its induced transformation under the shift.

The methods we use have their origin in Rauzy's classic
papers (\cite{R1}, \cite{R2}, \cite{R3})
and their subsequent development
by many authors.
Most significant and relevant here is the process
laid out by Canterini and Siegel (\cite{CS1}, \cite{CS2})
and from a somewhat different angle
by Holton and Zamboni (\cite{HZ1}, \cite{HZ2}).
The paper \cite{ABB} provides an excellent exposition of the process
and related constructions for a very important special case.
We adopt these single substitution methods to the S-adic case.
While often the generalizations are reasonably straightforward,
they differ in enough detail that an independent,
self-sufficient treatment is required.

The main idea in this geometric representation process 
 is projection onto the
stable subspace of the Abelianizations of the substitution.
For the  family \eqref{subdef} the Abelianizations are
\begin{equation}\label{subabel}
 A_n := \left[\begin{array}{ccc} 0 & n+1 & n\\
 1 & 0 & 0\\ 
0 & 1 & 1
\end{array}\right].
\end{equation}
Each $A_n$ has two eigenvalues outside the unit circle
and one inside (proof of Lemma 52 in \cite{beta}) 
and so each substitution \eqref{subdef}
is inverse-Pisot, unimodal and primitive.
It is shown in Theorem~\ref{mainconv} below that for 
each $\un\in\what$  the limit
\begin{equation*}\label{limabel}
\lim_{k\raw_\infty} A_{n_1} A_{n_2} \dots A_{n_k}
\end{equation*}
has a well defined asymptotic (or generalized)
one-dimensional stable direction
$\vell_\un$.
A similar argument  shows the existence of 
an asymptotic two-dimensional unstable subspace. However,
if the $n_i$ grow sufficiently fast there is not
an asymptotic one-dimensional strongest unstable direction (Theorems
11 and 12 in \cite{BT} and Theorem 27 in \cite{lexi}).
This implies that in these cases the asymptotic digit frequency vector of
$\rseq{\alpha}$ does not exist and, in addition, $\Dunp$ is
not uniquely ergodic (\cite{BT}).

To achieve the projection, Rauzy's original method was to embed
the Abelianization of the sequence $\rseq{\alpha}$  in $\R^3$
 and then project its vertices down to the stable subspace and take
the closure. We adapt the alternative route developed by
Canterini and Siegel (\cite{CS1}, \cite{CS2})
 and use  the machinery of the Prefix-Suffix Automaton.
 This automaton
is a particularly useful way of labeling and ordering  a 
Bratelli diagram based on a single substitution. 
The sequence of edge labels
in an infinite path naturally yield both the Dumont-Thomas
prefix-suffix expansion
of a symbolic sequence (\cite{DT}, \cite{DT2}) and a map to the
stable subspace of the Abelianization.
In the natural generalization to the S-adic case given in Example 3.5
of \cite{Sadic} and used below, each level of  the diagram
corresponds to a substitution from the list $\Lambda_{n_1}, 
\Lambda_{n_2}, \dots$. The sequence of edge labels $(u_i,a_i, v_i)$
in an infinite path  then yields  the  
 S-adic  version  
of the Dumont-Thomas prefix-suffix expansion
as in  formula (5) in \cite{Sadic} 
\begin{equation*}
\dots \Lambda_{n_2}\Lambda_{n_1}(u_2) \Lambda_{n_1}(u_1)\,
  u_0.a_0 v_0  \,  \Lambda^{(1)}(v_1) \Lambda_{n_2}\Lambda_{n_1}(v_2) \dots. 
\end{equation*}
With the appropriate
alterations for the S-adic situation, this yields
a map from the path space of the Bratelli diagram
to $\Dunp$ as well as a projection onto  the asymptotic
stable subspace of the infinite composition of  the Abelianizations
in \eqref{limabel}.
The inverse of the first map composed with the second map
yields the map $\Upsilon^+$ from $\Dunp$ to the real line 
given in Theorem~\ref{main} above.  The image
of $\Upsilon^+$ is sometimes called a generalized Rauzy fractal, and in
our case it is always a Cantor set embedded in an interval

The final ingredient in this
 geometric representation process is the subdivision
of the  Rauzy fractal 
into pieces corresponding to the images of symbolic
cylinder sets under $\Upsilon^+$.
The importance of these subpieces is that there is a
natural translation induced on them by the shift map
on $\Dunp$. Thus the process yields a  geometric representation
as long as the subpieces are disjoint almost
everywhere and so  proving this disjointness is often the
central problem in this process of geometric representation. 
For the infimax family, these subpieces are again
Cantor sets and we show in Theorem~\ref{inequality} that their
convex hulls (obviously intervals) can only 
intersect at their endpoints. The induced map on
these interval convex hulls is the representing Interval
Exchange map (ITM)
and we then show that the Rauzy fractal is, in fact,
the attractor of this ITM.

While the construction of an infimax minimal set requires
infinitely many choices to designate the list of defining substitutions in 
\eqref{alphaone}, one of the striking features of 
the geometric representation is that it 
faithfully describes an infimax minimal set by specifying just
two parameters in the family of ITM. The geometric representation
has additional consequences. For example, Boshernitzan and Kornfeld
note in \S 7 of \cite{BK} that after using the defining intervals
of the ITM to code
orbits it is straightforward to show
 that the number of distinct words of length
$n$ in itineraries of orbits can grow at most at a polynomial
rate. This implies that the ITM have zero topological entropy.
Thus using the geometric representation all the infimax minimal
sets also have zero entropy (this also follows from more general,
 more recent results; see Theorem 4.3 in \cite{Sadic}).
 Boshernitzan and Kornfeld
ask whether this factor complexity growth
rate is actually linear. In the special case of what are called
infimax minimal sets here, Cassaigne and Nicolas
showed that the factor complexity satisfies
$p(j) \leq 3j$  (\cite{CN}).

In many of the geometric representations of substitutions
in the literature the symbolic minimal set is represented
by either an interval exchange map or a toral translation.
This provides a great deal of information about the symbolic
minimal set, in particular, about its spectrum. Compared to interval
exchange maps, ITM are poorly understood and there is little
known about their spectrum. The ITM in the representing family
here have a rich variety of behaviours like non-unique
ergodicity and thus provide a good model problem for future development:
one can work jointly with the ITM and the symbolic, 
S-adic description.

There are a number of important features that are specific to the
family studied here. The first is that a 
list of substitutions $\Lambda_{n_1}, \Lambda_{n_2}, \dots $
when it acts on  bi-infinite sequences 
yields asymptotic period-two points rather than a fixed point.
As a consequence, the resulting Bratteli-Vershik diagram
is not properly ordered: it has  two maximal elements
and one minimal element (for background on 
Bratteli-Vershik diagrams and adic transformations see  \cite{HPS}
and for their use with substitutions \cite{VL} and \cite{DHS}).
 This complicates the
reading off the Dumont-Thomas expansion of a sequence from
the labels on edges of the maximal paths and necessitates
the eventual use of a map from $\Dunp$ back to the path
space in Section~\ref{mainmapsect}.

Also, the improper order
 implies that the Vershik map on the path space
cannot be globally defined. As done in \cite{VL}
and many subsequent papers, it can be defined  
almost everywhere 
but since we are working in the topological
category, we require all maps to be globally defined.
Our main objective is a map from $\Dunp$ to $\R$, and so
the path space is a convenient intermediate structure,
but we never need to consider the dynamics on it.
Thus the order on the 
Bratteli diagram and the Vershik map are not utilized here.
 However, the labeling of edges in the diagram by prefixes and suffixes
is of crucial importance
and so we have adapted the terminology  ``infinite prefix-suffix
automaton'' (IPSA) for the labeled Bratteli diagram corresponding
to a list of substitutions indexed by $\un\in\what$.

Another special feature  is the central role of the
lexicographic order on the symbol space and its relation
under the representation map to the usual linear order on
the real line. Note that each substitution $\Lambda_n$
preserves the lexicographic 
order (Lemma 2 in \cite{lexi})
and that the family arose as the solution to
the infimax question  which depends
fundamentally on the lexicographic order.
In the final analysis it 
is the relation of the orders on sequences and the reals
which yields the fundamental fact of
the disjointness of the subpieces of the Rauzy fractal.

Finally, in contrast to much of the existing work on S-adic
systems, the infimax family contains infinitely
many substitutions and further, our results hold for 
all sequences $\un\in\what$ rather than just 
a large, say full measure, subcollection. 

While our main results concern the one-sided shift space,
 a number of steps in the representation process are technically 
simpler using two-sided infinite sequences. This has the added
benefits of yielding useful  results about the relationship
between the
 two-sided version $\Dun$ and the one-sided version $\Dunp$ of
the infimax minimal set. For example, in Theorem~\ref{onesided}
we show that the projection $\pi:\Dun\raw\Dunp$ is
injective except on the forward orbits of the asymptotic
period two points, or  informally, the one-sided version
is obtained by collapsing a single pair of orbits of the two-sided
version.
This in turn implies that the left shift has unique inverses
in $\Dunp$ except on the 
$\rseq{\alpha}$ defined in \eqref{alphaone}.   

This paper deals primarily with the geometric representation
of the S-adic family with $N=3$ symbols: more detailed 
results about the infimax minimal
sets and their languages are saved for a later paper. 
We remark on the $N\not= 3$ case in the last
section.

\section{Preliminaries}
We start with some basic definitions about words, sequences and substitutions.
The alphabet here will always be  $\cA = \{1, 2, 3\}$.
The length of a finite word $w$ is denoted $|w|$, and 
the empty word $\epsilon$ has $|\epsilon| = 0$. A \textit{bi-infinite
sequence} $\us$ is an element of $\Sigma_3 = \cA^\Z$ and is written
with a decimal point between the zeroth and minus first
symbols, $\us = \dots s_{-2} s_{-1} . s_0 s_1 s_2 \dots$.
A \textit{right infinite sequence} has the form $s_0 s_1 \dots$
and we use an under-arrow to indicate it
 $\rseq{s} = s_0 s_1 \dots$, and a \textit{left infinite
sequence} is written $\lseq{s} = \dots s_{-2} s_{-1}$.
The collection of   right infinite sequences
is denoted  $\Sigma_3^+$. 

The spaces $\Sigma_3$ and $\Sigma_3^+$ are given the topology induced
by the metric
$
d(\us,\us') = 1/(1 + M)$ where 
$M = \min\{|i|\colon s_i \not= s'_i\}$.
The left shift $S$ acting on  $\Sigma_3$ is 
$S(\us) = \dots s_{-2} s_{-1}  s_0. s_1 s_2 \dots$
and acting on $\Sigma_3^+$ is  $S(\rseq{s}) = . s_1 s_2 \dots$. In the
dynamics literature $\sigma$ is usually used for the shift
and in  the substitutions literature $\sigma$ usually denotes
a substitution. To avoid confusion we refrain from using
$\sigma$ altogether.

A \textit{pointed word} is word with decimal point placed between
two of its symbols or at  the beginning or end of the word.
The shift acts on pointed words as long as its action does not
move the decimal point beyond the beginning or end of the word.
A \textit{pointed one-sided sequence} has the form 
$\dots s_{-2} s_{-1} . s_0 s_1 s_2 \dots s_n \epsilon$ or
$\epsilon s_{-n} \dots s_{-2} s_{-1} . s_0 s_1 s_2 \dots$. The
empty symbol $\epsilon$ is included to indicate the end
or beginning of the pointed one-sided sequence.  The shift
also acts on pointed one-sided sequences again with the
proviso that the decimal point cannot move beyond the
end.

A substitution $\Lambda$ is specified by assigning a nonempty word
$\Lambda(a)$ to each symbol $a\in\cA$. It acts on sequences,
words and pointed objects yielding 
another object of the same type by 
respecting the decimal point, so, for example,
 \begin{equation*}
\Lambda(\epsilon s_{-2} s_{-1}) . s_0
s_1 s_2 \dots = \epsilon \Lambda(s_{-2}) \Lambda(s_{-1}) . 
\Lambda(s_0) \Lambda(s_1) \Lambda(s_2) \dots.
\end{equation*}

For a \homeo\ $h\colon X\raw X$, the full orbit of
a point $x$ is \newline
$o(x, h) = \{ \dots, h^{-2}(x), h^{-1}(x), x,
h(x), h^2(x), \dots \}$ and its forward orbit is
$o^+(x, h) = \{  x,
h(x), h^2(x), \dots \}$. When $f:X\raw X$ is not injective,
$o^+(x,f)$ is defined in  the same way.

\section{The Infinite Prefix-Suffix Automaton}
\subsection{Definitions}
Fix a sequence $\un\in \what$. Our main object of study is
the sequence of substitutions $\Lambda_{n_1}, \Lambda_{n_2}, \dots$.
For each $k>0$, 
let 
\begin{equation*}
\Lambda^{(k)} = \Lambda_{n_1} \circ \Lambda_{n_2}
\dots\circ \Lambda_{n_k}.
\end{equation*}
 For a subcollection of indices
$n_j n_{j+1} \dots n_{j+k}$ write 
$\Lambda_{n_j n_{j+1} \dots n_{j+k}} =
\Lambda_{n_j} \circ \Lambda_{n_{j+1}}
\dots\circ \Lambda_{n_{j+k}}$, and so 
$\Lambda^{(k)} = \Lambda_{n_1 n_{2} \dots n_{k}}$. 

We now define the principal tool in this paper, the
\textit{Infinite Prefix-Suffix Automaton} (IPSA).  
A related automaton for a single substitution was contained in 
 Rauzy's classic
papers (\cite{R1}, \cite{R2}, \cite{R3}) and similar constructions 
were present in other seminal works in other
fields (see page 218 of \cite{S1} for some history). 
The automaton was  formalized, extended and utilized in 
\cite{CS1} and \cite{CS2} and independently in a slightly different form
in \cite{HZ2}. The version we use here is the S-adic generalization described
in  Example 3.5 of \cite{Sadic}.

The \textit{Infinite Prefix-Suffix Automaton} (IPSA)
or Bratteli-Vershik diagram
 associated with $\un\in \what$ is 
an infinite directed graph built in levels. Each level of the graph contains
three nodes or states 1, 2, 3 and the levels are indexed by 
$0, 1, 2, \dots$.  There is a directed edge from
state $a$ on level $n-1$ to state $b$ on level $n$
if and only if $\Lambda_n(b) = uav$ for some perhaps empty words 
$u$ and $v$. This edge is then labeled $(u, a, v)$. Note that
that level zero of our diagram has three states and not
a single root state as is common in the literature. 
 See Figure 1.

\begin{figure}[t]\label{fig1}
\centering
\includegraphics[width=0.5\textwidth]{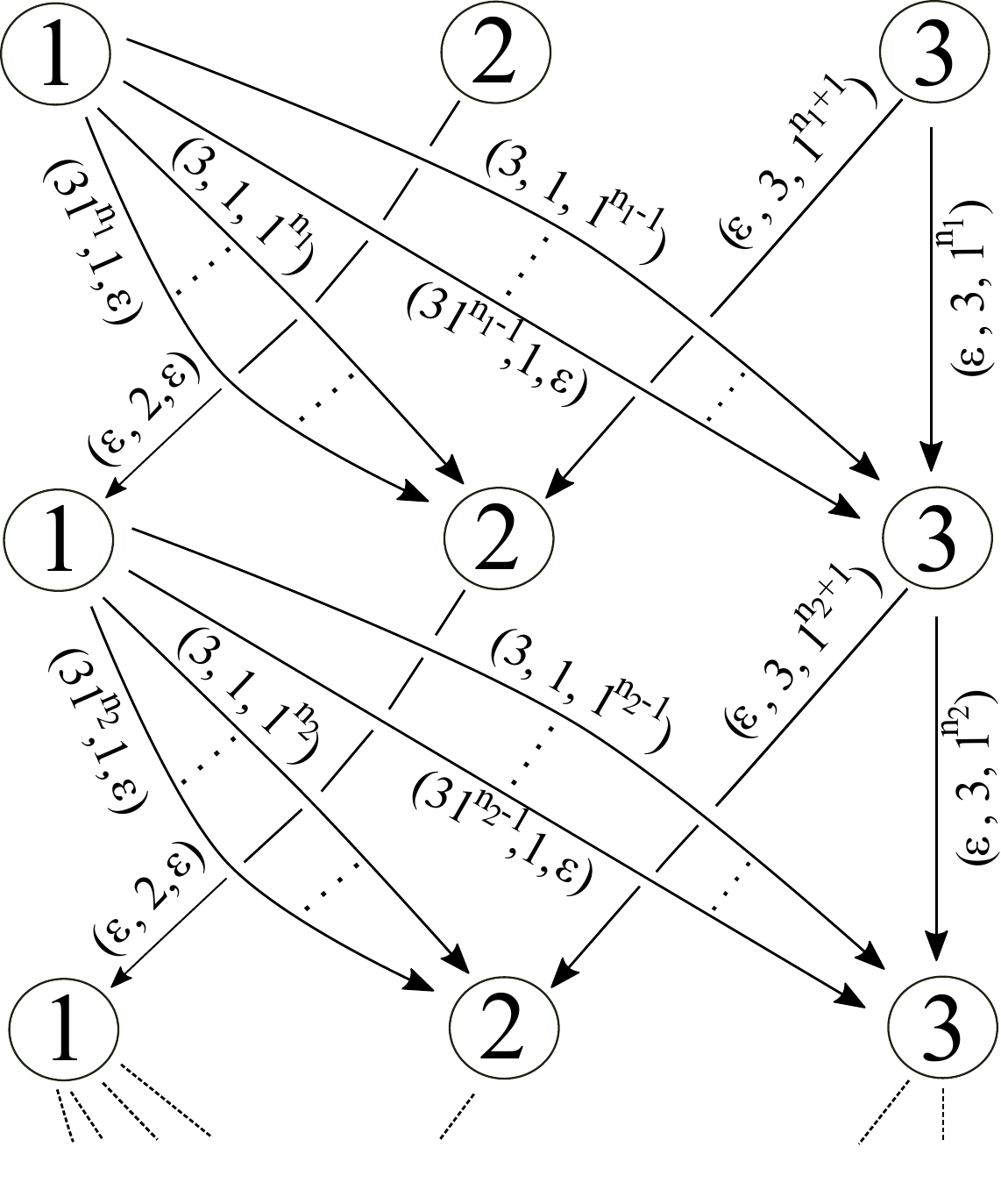}
\caption{First two levels of the IPSA}
\end{figure}
It is easy to check that an infinite path $\Gamma$ in the IPSA is uniquely
specified by its sequence of labels and we write
 $\Gamma =\Gamma_0 \Gamma_1 \dots$ or $\Gamma = 
(u_0, a_0, v_0), (u_1, a_1, v_1), \dots$ where for all $i$,
$\Lambda_{n_i}(a_i) = u_{i-1} a_{i-1} v_{i-1}$. Note that the 
sequence of edges is indexed by $0, 1, 2, \dots$ while
the sequence of indices $\un$ is indexed by $1, 2, 3, \dots$.
The collection of all infinite paths in the IPSA generated by
an $\un$ is denoted $\cP_\un$ or just $\cP$. The space $\cP$ is
a compact metric space under the metric 
$d(\Gamma, \Gamma') =  1/(N+1)$ where 
$N = \min\{i: \Gamma_i \not= \Gamma_i'\}$.

\subsection{The Dumont-Thomas expansion; the word and sequence maps}
A finite path
$\gamma = \Gamma_0, \Gamma_1 \dots, \Gamma_{k}$
in the IPSA is said to have length $k+1$ and the collection
of all such length $k+1$ paths is denoted $\cP^{(k+1)}$.
We now define the maps which assign a Dumont-Thomas expansion
 (\cite{DT}, \cite{DT2}) of a word or sequence to each finite
 or infinite path.

For each $k$ the
\textit{word map} assigns a finite pointed word to
a length $k+1$ path via
\begin{equation}\label{wordmapdef}
\cW(\gamma) = \Lambda^{(k)}(u_k) \dots \Lambda^{(1)}(u_1)\,
  u_0.a_0 v_0  \,  \Lambda^{(1)}(v_1) \dots \Lambda^{(k)}(v_k).
\end{equation}
Given an infinite path $\Gamma = \Gamma_0, \Gamma_1 \dots$, for each
$k$ let $\Gamma^{(k)}$ be the length $(k+1)$ path 
$ \Gamma^{(k)}= \Gamma_0, \dots, \Gamma_{k}$.
The \textit{sequence map} assigns a pointed  sequence
to each infinite path via
\begin{equation}\label{seqmapdef}
\cW(\Gamma) = \lim_{k\raw\infty} \cW(\Gamma^{(k)}) =
\dots \Lambda^{(2)}(u_2) \Lambda^{(1)}(u_1)\,
  u_0.a_0 v_0  \,  \Lambda^{(1)}(v_1) \Lambda^{(2)}(v_2) \dots 
\end{equation}
In most cases the sequence $\cW(\Gamma)$ will be bi-infinite,
but in certain special cases studied in Section~\ref{sectspec}
it will be infinite in only one direction. 

The first lemma describes in more detail
how the word map gives a correspondence between
paths in the IPSA  and symbolic words. As an example,
let $\Gamma_0 \Gamma_1 = (u_0, a_0, v_0), (u_1, a_1, v_1)$ be a length 
two path that terminates in the node labeled $a_2$. 
Thus from the definitions and keeping track of the decimal points  
we have $S^{j_1} \Lambda_{n_1}(\epsilon.a_1) = u_0.a_0 v_0 $ and
$S^{j_2} \Lambda_{n_2}(\epsilon.a_2) = u_1.a_1v_1 $ for
some $0\leq j_i < |\Lambda_{n_i}(a_i)|$.  Evaluating the
word map we have
\begin{equation*}
\begin{split}
\cW(\Gamma_0\Gamma_1) &= \Lambda_{n_1}(v_1) u_0.a_0.v_0 \Lambda_{n_1}(v_1)
= \Lambda_{n_1}(v_1) S^{j_1} \Lambda_{n_1}(\epsilon.a_1) \Lambda_{n_1}(v_1)\\
&= S^{j_1}\Lambda_{n_1}(u_1.a_1v_1) = 
S^{j_1}\Lambda_{n_1}(S^{j_2} \Lambda_{n_2}(\epsilon.a_2))
=S^{j_1+j_2}\Lambda^{(2)}(\epsilon.a_2)
\end{split}
\end{equation*}
Thus for 
any length two path which terminates at the node
$a_2$, its word map image is a shift of $\Lambda^{(2)}(\epsilon.a_2)$. In fact,
Lemma~\ref{form} and Theorem~\ref{assignment} together will
show that each shift of $\Lambda^{(2)}(\epsilon.a_2)$ corresponds to
  one and only one length two path terminating at $a_2$.
In general,  
let $\cP^{(k)}(a)$ be the subcollection of length-$(k)$ paths which 
terminate at node $a$ and so $\Lambda_{n_{k}}(a) = u_{k-1} a_{k-1} v_{k-1}$. 
The first lemma shows that for any path in $\cP^{(k)}(a)$ its word
map image is a shift of $\Lambda^{(k)}(\epsilon.a)$. Theorem~\ref{assignment}
will show that the correspondence is injective. Part (b)
of the lemma shows how to replace the inner portion
of a  sequence map image with the form of the image of an initial
segment of the path given in part (a).   
\begin{lem}\label{form} For all $\un\in \what$,
\begin{enumerate}[(a)]
\item if $\gamma \in \cP^{(k)}(a)$, then
$\cW(\gamma)= S^j(\Lambda^{(k)}(\epsilon.a))$ for some 
 $0 \leq j < |\Lambda^{(k)}(a)|$.
\item If $\Gamma\in\cP$, then for all $k>0$ there exists 
$0 \leq j < |\Lambda^{(k)}(a_k)|$ so that
 $$\cW(\Gamma) = \dots \Lambda^{(k+1)}(u_{k+1})\, \Lambda^{(k)}(u_k)\,
S^j(\Lambda^{(k)}(\epsilon.a_k))\,  \Lambda^{(k)}(v_k)\, \Lambda^{(k+1)}(v_{k+1})\dots
$$ 
where $S^j(\Lambda^{(k)}(\epsilon.a_k)) = \cW(\Gamma^{(k-1)})$.
\end{enumerate}
\end{lem}

\begin{proof}
We prove (a) by induction on $k$, with  the case $k=0$ following
from the definition of the IPSA. So assume the result is true
for length $k$ paths.
Given $\gamma = \Gamma_0, \Gamma_1, \dots, \Gamma_{k}$
with $a_{k+1} = a$, 
we first show that $\cW(\gamma)$
has the required form. 
Form the
length $k$ path 
$\gamma' =  \Gamma_1, \Gamma_2, \dots, \Gamma_{k}$ in the IPSA of
$\un' = n_2 n_3 \dots$. Using the inductive hypothesis 
on $\cW'$, the assignment associated with $\un'$, we have
a $0 \leq j' < |\Lambda_{n_2 \dots n_{k}}(a)|$ with
\begin{align*}
S^{j'}(\Lambda_{n_2 \dots n_{k+1}}(\epsilon.a)) &= \cW'(\gamma')\nonumber\\ 
&= \Lambda_{n_2 \dots n_k}(u_k) \dots \Lambda_{n_2}(u_2)\, u_1.a_1 v_1 
\, \Lambda_{n_2}(v_2) \dots \Lambda_{n_2 \dots n_k}(v_k).
\end{align*}
Thus
\begin{equation*}
\cW(\gamma) = S^{|u_0|} (\Lambda_{n_1} \cW'(\gamma')) 
= S^{|u_0|} (\Lambda_{n_1} S^{j'}(\Lambda_{n_2 \dots n_{k}}(\epsilon.a)))
= S^{j}(\Lambda^{(k)}(\epsilon.a)),
\end{equation*}
for some $j\geq 0$ where since we know  that $\cW(\gamma)$ is a finite, pointed
word, $ j\leq |\Lambda^{(k)}(a)|$. But since
from the IPSA, $a_0 = \epsilon$ is impossible, in fact
$j < |\Lambda^{(k)}(a)|$, finishing (a).

For (b) it follows immediately from the definition that 
for any $k>0$, 
$$\cW(\Gamma) = \dots \Lambda^{(k)}(u_k) \cW( \Gamma^{(k-1)} )
\Lambda^{(k)}(v_k) \dots,$$
and so (b) then follows from (a).
\end{proof} 
\begin{rem}\label{negj}
Letting $j' =  |\Lambda^{(k)}(a_k)| - j$, we also have
 $$\cW(\Gamma) = \dots \Lambda^{(k+1)}(u_{k+1})\, \Lambda^{(k)}(u_k)
S^{-j'}\Lambda^{(k)}(a_k.\epsilon) \,
 \Lambda^{(k)}(v_k) \Lambda^{(k+1)}(v_{k+1})\dots
$$
\end{rem}

\subsection{Some special paths and sequences}\label{sectspec}
The sequence map given in \eqref{seqmapdef} yields a  two-sided
sequence for most paths in $\cP$. The paths $\Gamma$ for which $\cW(\Gamma)$
is a pointed one-sided sequence require individual consideration 
and fall into three classes. Recall that two paths $\Gamma$ and $\Gamma'$
are said to be \textit{tail equivalent} if there exists a $k$ so that
$\Gamma_i = \Gamma_i'$ for all $i\geq k$. This obviously yields
an equivalence relation on the path space $\cP$.
\begin{itemize}
\item The set $\spec_1$ will consist of the tail equivalence class of
\begin{align}\label{speconedef}
\Gamma_\beta &:= (3 1^{n_{1}}, 1,\epsilon), 
 \prod_{i=1}^\infty ((\epsilon, 2, \epsilon), 
(3 1^{n_{2i + 1}}, 1,\epsilon))\nonumber\\
 & \text{or all}\  \Gamma = \Gamma_0, \dots, \Gamma_w,
 \prod_{i=1}^\infty ((\epsilon, 2, \epsilon), 
(3 1^{n_{w + 2i + 1}}, 1,\epsilon)) 
\end{align}
with $\Gamma_w\not = ( 3 1^{n_{w+1}}, 1, \epsilon)$
and $w$ even.
\item The set $\spec_\hone$ will consist of the tail equivalence class of 
\begin{align}\label{spechonedef}
\Gamma_\hbeta &:= 
 \prod_{i=1}^\infty  ((\epsilon, 2, \epsilon), 
(3 1^{n_{2i}}, 1,\epsilon))\nonumber\\
& \text{or all}\  \Gamma = \Gamma_0, \dots, \Gamma_w,
 \prod_{i=1}^\infty \left((\epsilon, 2, \epsilon), 
(3 1^{n_{w + 2i + 1}}, 1,\epsilon)\right) 
\end{align}
with $\Gamma_w\not = ( 3 1^{n_{w+1}}, 1, \epsilon),$
and $w$ odd.
\item The set $\spec_2$ will consist of the tail equivalence class of
\begin{equation}\label{spectwodef}
\Gamma_\alpha :=  \prod_{i=1}^\infty (\epsilon, 3, 1^{n_i})
\ \text{or all}\   \Gamma = \Gamma_0, \dots, \Gamma_w,
 \prod_{i=w+1}^\infty (\epsilon, 3, 1^{n_i})
\end{equation}
with $\Gamma_w\not = (\epsilon, 3, 1^{n_{w+1}})$.
\item Finally,
$\cN = \cP\setminus(S_1\cup S_\hone\cup S_2)$. Note that
$\Gamma\in\cN$ if and only if there exist arbitrarily large $k$ with
$u_k\not= \epsilon$ and 
 there exist arbitrarily large $k$ with
$v_k\not= \epsilon$. Thus $\cN$ consists of those paths
for which $\cW(\Gamma)$ is a two-sided infinite sequence.
\end{itemize}

 As mentioned in the introduction, we don't make
explicit use of the substitution induced 
order on the path space but its description will clarify what follows.
On a  path space derived from a general diagram, Vershik defines
a partial order which restricts to a total order on each tail equivalence
class. It is built from a total order on the outgoing edges
at each vertex. Specifically, for two paths in the same
tail equivalence class, the lesser one is the one with the lesser
outgoing edge at the last vertex they disagree. The Vershik
map on the path space sends each non-maximal path to its
successor. 

In the particular case of a Bratteli-Vershik diagram which
is the Prefix-Suffix Automaton for a single substitution $\sigma$, it
was noted in \cite{CS1} and \cite{HZ2} that there is  a natural total order 
on the outgoing edges from a vertex. If the vertex
is labeled by the letter $b$, each outgoing edge is
labeled $(u,a,v)$ where $\sigma(b) = uav$, and so declare
$(u,a,v) < (u',a',v')$ if $|u| < |u'|$. When the path space
is mapped onto the substitution minimal set via the analog
of \eqref{seqmapdef}, the Vershik map on the path space 
will conjugate (on a large set) 
with the shift map on the substitution minimal set.
 
For the IPSA of the infimax S-adic family considered here,
one may define a partial order on the path space exactly as in the
single substitution case. In this order the path
$\Gamma_\alpha$ is the maximal
element and the paths $\Gamma_\beta$ and $\Gamma_\hbeta$ are
minimal elements. Thus the sets $\spec_i$ just defined are the
tail equivalence classes of the maximal and minimal elements.
We therefore expect from Vershik's construction that each set
would correspond to an orbit of sequences under the shift.
This is the content of Lemma~\ref{specprop}. The corresponding
sequences are built from the following:
\begin{align}\label{alphabeta}
\rseq{\alpha}  &= \lim_{k\raw\infty} \Lambda^{(k)}(\epsilon.3)\nonumber\\ 
\lseq{\beta} &= \lim_{k\raw\infty} \Lambda^{(2k)}(1.\epsilon) =
 \lim_{k\raw\infty} \Lambda^{(2k-1)}(2.\epsilon)\\ 
\lseq{\hbeta}  &= \lim_{k\raw\infty} \Lambda^{(2k)}(2.\epsilon) =
\lim_{k\raw\infty} \Lambda^{(2k+1)}(1.\epsilon)\nonumber
\end{align}
It is clear that the limits exist.
Let $\ut = \lseq{\beta}.\rseq{\alpha}$ and 
$\uht = \lseq{\hbeta}.\rseq{\alpha}$. These are the
S-adic analogs of the period two point in the single
substitution case in the sense that
\begin{align*}
\text{If}\ s_{-1}=1, \ &\text{then}\ \Lambda^{(2k)}(\us)\raw \ut\ 
\text{and} \ \Lambda^{(2k+1)}(\us)\raw \uht\\
\text{If}\ s_{-1}=2\ \text{or}\ 3, \ &\text{then}\
\Lambda^{(2k+1)}(\us)\raw \ut\ 
\text{and} \ 
\Lambda^{(2k)}(\us)\raw \uht
\end{align*}

\begin{lem}\label{specprop} For an $\un\in \what$,
\begin{enumerate}[(a)]
\item if $\Gamma\in\spec_1$ as in \eqref{speconedef}, 
there exists a $j< 0$ with   $\cW(\Gamma) = 
S^j(\lseq{\beta}.\epsilon)$.  In particular, $\cW(\Gamma_\beta)
= S^{-1}(\lseq{\beta}.\epsilon)$.
\item 	 If $\Gamma\in\spec_\hone$ as in \eqref{spechonedef},  
there exists a $j< 0$ with   $\cW(\Gamma) = S^j(\lseq{\hbeta}.\epsilon)$. 
In particular, $\cW(\Gamma_\hbeta)
= S^{-1}(\lseq{\hbeta}.\epsilon)$.
\item 
If $\Gamma\in\spec_2$ as in \eqref{spectwodef}, 
there exists a $j\geq 0$ with $\cW(\Gamma) = 
S^j(\epsilon.\rseq{\alpha})$.  In particular, 
$\cW(\Gamma_\alpha) = \epsilon.\rseq{\alpha}$
\end{enumerate}
\end{lem}
\begin{proof}
We start with (c). If $\Gamma\in\spec_2$ as in \eqref{spectwodef},
then $a_k = 3$ for all $k>w$. Thus by Lemma~\ref{form}(b), there
is some $j = j(k)$ with  
$\cW(\Gamma) = \dots S^{j(k)}(\Lambda^{(k)}(\epsilon.3)) \dots$. 
On the other hand,
$u_k = \epsilon$ for all $k > w$, and so for those $k$, $j(k)$ is
 a constant $j$, and thus $\cW(\Gamma) = 
\lim_{k\raw\infty} \epsilon\, S^j(\Lambda^{(k)}(\epsilon.3)) \dots = 
S^j(\epsilon.\rseq{\alpha})$.
In particular since for $\Gamma_\alpha$ all $u_k = \epsilon$, 
$\cW(\Gamma_\alpha) = \epsilon.\rseq{\alpha}$.

The arguments for (a) and (b) are similar and we just give (a).
If $\Gamma\in\spec_1$ as in \eqref{speconedef},
then $a_{2m} = 1$ for all $2m>w$. Thus by Remark~\ref{negj}, there
is some $j = j(m)\geq 1$ with  
$\cW(\Gamma) = \dots S^{-j(m)}\Lambda^{(2m)}(1.\epsilon) \dots$. 
On the other hand,
$v_k = \epsilon$ for all $k > w$, and so for those $k$, $j(k)$ is
 a constant $j$, and thus $\cW(\Gamma) = 
\lim_{m\raw\infty}\dots S^{-j}(\Lambda^{(2m)}(1.\epsilon))\, \epsilon
 = S^{-j}(\lseq{\beta}.\epsilon)$.
In particular since for $\Gamma_\beta$ all $v_k =\epsilon$ and
$a_0 = 1$, so 
$\cW(\Gamma_\beta) = S^{-1}(\lseq{\beta}.\epsilon)$.
\end{proof}

With Lemma~\ref{specprop} in  mind, now define
\begin{align}
\sspec_1 &= \{S^j(\lseq{\beta}.\epsilon)\colon j<0 \}\nonumber\\
\sspec_\hone &= \{S^j(\lseq{\hbeta}.\epsilon)\colon j<0 \}\\
\sspec_2 &= \{S^j(\epsilon.\rseq{\alpha})\colon j\geq 0 \}\nonumber
\end{align}

\subsection{Injectivity of the word and sequence maps}
The  maps $\cW$ takes the prefix-suffix labels along
a path and generates a word or sequence via the Dumont-Thomas expansion.
The next theorem states that this assignment is injective. 
\begin{thm}\label{assignment}
Given $\un\in\what$.
\begin{enumerate}[(a)]
\item For each $a \in\{1,2,3\}$,
 the word map $\gamma\mapsto \cW(\gamma)$ defined in \eqref{wordmapdef} 
gives a bijection between $\cP^{(k)}(a)$ and
\begin{equation*}
\cS^{(k)}(a) :=\{S^j (\Lambda^{(k)}(a)) \colon 0 \leq j < |\Lambda^{(k)}(a)|\}.
\end{equation*}
\item The sequence map $\Gamma\mapsto \cW(\Gamma)$  
defined in \eqref{seqmapdef}
 is an injection on $\cN$  and
 a bijection between $\spec_i$ and $\sspec_i$ for $i = 1, \hone, 2$.
\end{enumerate}
\end{thm}

\begin{rem}
This theorem coupled with later results may be be viewed as a desubstitution
or recognizability result. As
in the proof of Lemma~\ref{form}, given $\un\in\what$ and a path
$\Gamma= \Gamma_0, \Gamma_1, \dots$ in $\cP_\un$,  let 
 $\un' = n_2 n_3\dots$ and form the   path 
$\Gamma' =  \Gamma_1, \Gamma_2, \dots \in \cP_{\un'}$.
If $\cW'$ is the sequence map associated with the path space
 $\cP_{\un'}$, then directly from the definitions of the sequence maps
\begin{equation*}
S^{|u_0|} \Lambda_{n_1}(\cW'(\Gamma')) = \cW(\Gamma).
\end{equation*}
Thus if $\cW(\Gamma)$ is a bi-infinite sequence (i.e. $\Gamma\in\cN$),
it is desubstituted under $\Lambda_{n_1}$ by $\cW'(\Gamma')$. By 
Theorem~\ref{assignment}, the assignment $\Gamma\raw\cW(\Gamma)$ is injective,
so the desubstitution is unique. To obtain a full result, we need
in addition that $\cW(\Gamma)\in\Delta_\un$ (Theorem~\ref{mainmap})
and for $\Gamma\in\cS_i$, we must extend $\cW(\Gamma)$ to
a bi-infinite sequence in $\Dun$ as described above  Theorem~\ref{mainmap}.
Note that a sequence in $\Dun$ is  desubstituted under $\Lambda_{n_1}$
by a sequence in $\Delta_{\un'}$. One may continue, desubstituting
under $\Lambda_{n_2}$ a sequence in $\Delta_{\un'}$ by one in
$\Delta_{\un''}$, where $\un'' = n_3 n_4 \dots$, etc. Fisher studies
this situation in \cite{fisher} and constructs the S-adic analog
of the subshift of finite type corresponding to a single
substitution (see Section 7.2 in \cite{S1}). 
\end{rem}

The proof of Theorem~\ref{assignment} 
requires more detailed  information on the word
and sequence maps $\cW$.
Since we are considering words and sequences of different types
we need a stronger notion of inequality.
Given words or one-sided or two-sided sequences $\us$ and $\us'$, they 
are said to be \textit{strongly unequal}, $\us\sne\us'$,
 at index $j$ if  $s_j \not=\epsilon$,
$s'_j \not=\epsilon$, and $s_j \not=  s'_j$.
Thus $1.233$ is not equal to $1.23$ but it is not strongly
unequal.
\begin{rem}\label{pairs}
 The following is easy to check from the IPSA, but very useful.
If $\cW(\Gamma)$ is bi-infinite, only the following pairs can
occur in the sequence $\cW(\Gamma)$:
 $11$, $12$, $13$, $22$, $23$ and $31$. The same restrictions hold
inside any $\cW(\Gamma)$ or  $\cW(\gamma)$, and in addition,
a right-one sided sequence or finite word can only start with a
$3$ while  a left-one sided sequence or finite word can only end with
a $1$ or a $2$.
\end{rem}
 
\begin{lem}\label{stronginjective}  For all $\un\in \what$:
\begin{enumerate}[(a)]
\item If $\Gamma\not=\barGamma$ in $\cP_{\un}$ or
$\gamma\not=\bargamma$ in $\cP^{(k)}$, then 
 $\cW(\Gamma) \sne \cW(\barGamma)$ and  $\cW(\gamma) \sne \cW(\bargamma)$.
\item If  $\Gamma\not=\barGamma$ in $\cP_{\un}$ 
then $\cW(\Gamma) \sne \cW(\barGamma)$ at an index $j\geq 0$
 except for the case of pairs of the form
\begin{align}\label{badpair1}
\Gamma &= \Gamma_0, \ldots \Gamma_{w-2}, (31^{n_{w -\ell}}, 1, 1^\ell), 
\Gamma_w, \ldots\\
\barGamma &= \Gamma_0, \ldots \Gamma_{w-2}, (31^{n_{w -\bell}}, 1, 1^\bell), 
\Gamma_w, \prod_{i=1}^\infty 
\left((\epsilon, 2, \epsilon), (31^{n_{w + 2i}}, 1, \epsilon)\right)\nonumber
\end{align}
with $\ell > \bell$ or
\begin{align}\label{badpair2}
\Gamma &= \Gamma_0, \ldots \Gamma_{w-2}, (31^{n_{w -\ell}}, 1, 1^\ell), 
\Gamma_w, \ldots\\
\barGamma &= \Gamma_0, \ldots \Gamma_{w-2}, (31^{n_{w -\bell-1}}, 1, 1^\bell), 
\Gamma_w, \prod_{i=1}^\infty 
\left((\epsilon, 2, \epsilon), (31^{n_{w + 2i}}, 1, \epsilon)\right)\nonumber
\end{align}
with $\ell\geq\bell$.
  
\end{enumerate}
\end{lem}

\begin{proof}
We prove (b) first.  The proof is by induction on $k$, 
the smallest index with $\Gamma_k \not=
\barGamma_k$. To start assume that $\Gamma_0 \not= \barGamma_0$. 
Now if $a_0 \not=\bara_0$, then $\cW(\Gamma) \sne \cW(\barGamma)$ at
the index $j=0$ and so we may assume $a_0 =\bara_0$. Since just
one edge emerges from state $2$, the case $a_0 = \bara_0 = 2$
is impossible, so we are left with two remaining cases.

Now if  $a_0 =\bara_0 = 3$ and $\Gamma_0 \not= \barGamma_0$,
  the only possibilities are, say,
$\Gamma_0 = (\epsilon, 3, 1^{n_1 + 1})$ which can only be followed by 
$\Gamma_1 = (\epsilon, 2, \epsilon)$, and 
$\barGamma_0 = (\epsilon, 3, 1^{n_1})$ which can only be followed by 
$\bar\Gamma_1 = (\epsilon,3, 1^m)$ with $m = n_2$ or $m= n_2 + 1$.
In either case, $\cW(\Gamma_0\Gamma_1)_R = .3 1^{n_1+1}$ and
$\cW(\barGamma_0\barGamma_1)_R = .3 1^{n_1}\Lambda_{n_1}(1^m) =
\epsilon.3 1^{n_1}2^m$, and so $\cW(\Gamma_0\Gamma_1) \sne 
\cW(\barGamma_0\barGamma_1)$ at index $j=n_1+1>0$.

Now assume  $a_0 =\bara_0 = 1$ and we consider various subcases. 
First assume that $\Gamma_0$ and $\barGamma_0$ both terminate at state
$3$. Thus $\Gamma_0 = (31^{n_1 -\ell-1}, 1, 1^\ell)$ and
$\barGamma_0 = (31^{n_1 -\bell-1}, 1, 1^\bell)$ where, say, 
$\ell > \bell$. Thus 
$\cW(\Gamma_0, \Gamma_1)_R = .1^{\ell+1}\Lambda^{(1)}(1^m)$ and
$\cW(\barGamma_0, \barGamma_1)_R = .1^{\bell+1}\Lambda^{(1)}(1^\barm)$ 
with $m, \barm = n_2$ or $n_2 + 1$. Since $\Lambda^{(1)}(1) = 2$,
$\cW(\Gamma)\sne\cW(\barGamma)$ at index $j=\bell + 1>0$.

Now assume that $\Gamma_0$ and $\barGamma_0$ both terminate at state
$2$. Thus $\Gamma_0 = (31^{n_1 -\ell}, 1, 1^\ell)$ and
$\barGamma_0 = (31^{n_1 -\bell}, 1, 1^\bell)$ where, say, 
$\ell > \bell$. If $\barv_k = \epsilon$ for all $k\geq 1$ 
then we are in the excluded case \eqref{badpair1}. Thus
for some $k\geq 1$, $\barv_k = 1^\barm$ with $\barm>0$. Thus
$\cW(\barGamma_0, \barGamma_1)_R = .1^{\bell+1}\Lambda^{(k)}(1^\barm)$
and since  $\Lambda^{(k)}(1^\barm)$ begins with either
a $2$ or a $3$, $\cW(\Gamma)\sne\cW(\barGamma)$ at index $j=\bell + 1>0$.
 
For the last subcase, assume that $\Gamma_0$ terminates at
$3$ and  $\barGamma_0$ terminates at $2$. Thus
 $\Gamma_0 = (31^{n_1 -\ell-1}, 1, 1^\ell)$ and 
$\barGamma_0 = (31^{n_1 -\bell}, 1, 1^\bell)$. Of necessity
$v_1 = 1^m$ with $m =  n_2$ or $n_2 + 1$, thus if
$\ell < \bell$, then $\cW(\Gamma)\sne\cW(\barGamma)$ at index 
$j=\bell + 1 > 0$. Now assume that $\ell \geq \bell$.
If $\barv_k = \epsilon$ for all $k\geq 1$ 
then we are in the excluded case \eqref{badpair2}. Thus
for some $\bark\geq 1$, $\barv_\bark = 1^\barm$ with $\barm>0$. In fact,
$\bark\geq 2$ because of necessity, $\barGamma_1 = (\epsilon, 2, \epsilon)$.
Since $v_1   = 1^m$ with $m =  n_2$ or $n_2 + 1$,
$\cW(\Gamma)_R = .1^{\ell + 1} \Lambda^{(1)}(1^m) = 
.1^{\ell + 1} 2\ldots$ and 
$\cW(\barGamma)_R = .1^{\bell + 1} \Lambda^{(\bark)}(1^\barm) = 
.1^{\bell + 1} 3\ldots$ using the fact that $\bark\geq 2$ and
so $\cW(\Gamma)\sne\cW(\barGamma)$ at index $j=\bell + 1$.  
 This finishes the initial inductive step of $k=0$.

Now assume the result is true for $k>0$. Let $\un' = n_2 n_3 \dots$ and
consider the infinite paths in the associated  IPSA, $\cP_{\un'}$, obtained
by deleting the first edges of $\Gamma$ and $\barGamma$ and so
$\Gamma' =  \Gamma_1, \Gamma_2, \dots$ and
$\barGamma' =  \barGamma_1, \barGamma_2, \ldots$. If $\Gamma$ and
$\barGamma$ are not of the form \eqref{badpair1} or \eqref{badpair2},
then neither are  $\Gamma'$ and $\barGamma'$. 
 Since $\barGamma'$ and $\Gamma'$ have $\barGamma'_k \not = \Gamma'_k$,
if $\cW'$ is the word map for $\cP_{\un'}$, then by  the inductive hypothesis  
$\cW'(\Gamma') \sne \cW'(\barGamma')$ at some index $j\geq 0 $ and
let $j'$ be the least such $j$. Now note 
as in the proof of Lemma~\ref{form},
$S^{|u_0|}\Lambda_{n_1} \cW'(\Gamma') =  \cW(\Gamma)$ and
$S^{|u_0|}\Lambda_{n_1} \cW'(\barGamma') =  \cW(\barGamma)$.
We claim that this implies that $\cW(\Gamma) \sne \cW(\barGamma)$
at an index $J\geq 0$. We first show the result that 
$\Lambda_{n_1} \cW'(\Gamma') \sne \Lambda_{n_1} \cW'(\barGamma')$
at an index $j \geq 0$ and then argue that, in fact, $j \geq |u_0|$

Now if it was the case that $\cW'(\Gamma') \sne \cW'(\barGamma')$
at an index $j'\geq 0$ 
with either $\cW'(\Gamma')_{j'} =1$ or $ \cW'(\barGamma')_{j'}=1$,
then the result is immediate since $\Lambda_{n_1}(1) = 2$ and both
$\Lambda_{n_1}(2)$ and $\Lambda_{n_1}(3)$  begin with a $3$.
So the only case remaining is, say,
 $\cW'(\Gamma')_{j'} =2$ and $ \cW'(\barGamma')_{j'}=3$.
Using Remark~\ref{pairs}, the $3$ must be followed
by a $1$. 
Since $j' \geq 0$  
and $\Lambda_{n_1} (2) = 3 1^{n_1+1}$
while $\Lambda_{n_1}(31)  = 3 1^{n_1} 2 $, we have the result
in this case also.
 
Next we show that $j\geq |u_0|$. When $k>1$, 
\begin{align*}
\Lambda_{n_1} \cW'(\Gamma')_R &= \Lambda_{n_1}(a_1) \Lambda^{(1)}(v_1)
\dots  \Lambda^{(k-1)}(v_{k-1}) \Lambda^{(k)}(v_{k})\dots\\
\Lambda_{n_1} \cW'(\barGamma')_R &= \Lambda_{n_1}(a_1) \Lambda^{(1)}(v_1)
\dots  \Lambda^{(k-1)}(v_{k-1}) \Lambda^{(k)}(\barv_{k})\dots
\end{align*}
and so whatever $j$ is it satisfies 
$$j \geq |\Lambda_{n_1}(a_1)| 
+ \sum_{i=1}^{k-1} |\Lambda^{(i)}(v_i)| = 
|u_0 a_0 v_0| 
+ \sum_{i=1}^{k-1} |\Lambda^{(i)}(v_i)|.
$$ 
When $k=1$, since $\Gamma_0 = \barGamma_0$, both 
$\Gamma_1$ and $\barGamma_1$, emerge from the same state
$a_1 = \bara_1$ and so 
\begin{align*}
\Lambda_{n_1} \cW'(\Gamma')_R &= \Lambda_{n_1}(a_1) \Lambda^{(1)}(v_1)
\dots  \\
\Lambda_{n_1} \cW'(\barGamma')_R &= \Lambda_{n_1}(a_1) \Lambda^{(1)}(\barv_1)
\dots  
\end{align*}
so $j\geq |\Lambda_{n_1}(a_1)| = |u_0 a_0 v_0|$.
Thus in every case, $j>|u_0|$ and so 
$\cW(\Gamma) \sne \cW(\barGamma)$
at an index $J= j-|u_0|\geq 0$

For (a),  note that for $\Gamma$ and $\barGamma$ as in
\eqref{badpair1} or \eqref{badpair2} have $u_i = \baru_i$ 
for $i = 1, \dots, w-2$ and $u_{w-1} \not= \baru_{w-1}$, 
thus $\cW(\Gamma) \sne \cW(\barGamma)$
at an index $j < 0$. 
\end{proof}

\noindent\textit{Proof of Theorem~\ref{assignment}:} 
For (a) first note the assignment 
is injective by Lemma~\ref{stronginjective}(a) and 
the inclusion $\cW(\cP^{(k)}(a)) \subset
\cS^{(k)}(a)$ follows from Lemma~\ref{form}.
To finish part (a),  by the definition of the IPSA, the Abelianizations and
standard graph theory, 
\begin{equation*}
\# (\cP^{(k)}(a)) = \| A_{n_1 \dots n_{k}}(\ve_a) \|_1 =
 |\Lambda^{(k)}(a)|,
\end{equation*}
and so the assignment is onto.

To see that $\cW:\spec_2\raw \sspec_2$ is onto, pick $S^j(.\rseq{\alpha})$
for some $j\geq 0$. Using part (a), for any $k$ with 
$0\leq j < |\Lambda^{(k)}(3)|$ there is a path $\gamma = \Gamma_0, \dots, \Gamma_k$
with $a_k = 3$ and $\cW(\gamma) = S^j \Lambda^{(k)}(3)$. Thus if we define the
infinite path $\Gamma =   \Gamma_0, \dots, \Gamma_k, 
\prod_{i={k+1}}^\infty (\epsilon, 3, 1^{n_i})$, then as in the proof of
Lemma~\ref{specprop},  
$\cW(\Gamma) = S^j(.\rseq{\alpha})$.  The proof that
$\cW:\spec_1\raw \sspec_1$ and $\cW:\spec_\hone\raw \sspec_\hone$
are onto are similar.
Since $\cW(\cN), \sspec_1, \sspec_\hone$ and  $\sspec_2$ are 
all mutually disjoint, the injectivity of all the assignments follow
from Lemma~\ref{stronginjective}.
$\Box$

\section{A topological lemma}
Recall that our goal is to construct a geometric representation
of the S-adic symbolic minimal sets defined in the introduction.
The word map $\cW$ takes us from paths to sequences and so we need a map in the
other direction. 
 This requires an elementary, but technical
topology lemma. For a subset $Z_1$ in a topological space $Z_2$,
its set of limit points is denoted $\LP(Z_1)$
\begin{lem}\label{toplem}
Assume $X$ and $Y$ are compact metric spaces with $X = X_1\sqcup X_2$
and there exist injections $f_1:X_1\raw Y$ and $f_2, f_3:X_2\raw Y$
so that when $Y_i = \image(f_i)$ for $i=1,2,3$  the $Y_i$ are disjoint
with $Y_1$ dense in $Y$. Define a set-valued function $\barf$ by
\begin{align*}
\barf(x) &= \{f_1(x)\} \ \text{when}\ x\in X_1\\
     &=  \{f_2(x), f_3(x) \} \ \text{when}\ x\in X_2.
\end{align*}
and assume that $\barf$ has the property that if $x_n\raw x\in X$,
then $\LP\{\barf(x_n)\} \subset \barf(x)$.
Then $\barf$ is onto in the sense that 
$Y = \cup_{x\in X} \barf(x)$ and if $g:Y\raw X$ is defined by
$g(y) = f_i\Inv(y)$ when $y\in Y_i$, then $g$ is continuous,
injective on $Y_1$ and two-to-one on $Y_2\cup Y_3$.
\end{lem}
\begin{proof}
We first prove $Y = \cup_{x\in X} \barf(x)$.
Since $Y_1 = f_1(X_1)$ is dense in $Y$ by hypothesis, for all $y\in Y$
there is a sequence $\{x_n\}\subset X_1$ with $f_1(x_n)\raw y$. Since
$X$ is compact, there is an $x_0$ and a subsequence $x_{n_i}\subset X_1$ with
$x_{n_i}\raw x_0$. Thus 
$\barf(x_0)\supset \LP\{\barf(x_{n_i})\} = \LP\{f_1(x_{n_i})\} = y$
and so $y = f_i(x_0)$ for some $i=1,2,3$ as required.

We now prove that $g$ is continuous. For $y\in Y_2$, let
$\hy = f_3\circ f_2\Inv(y)$ and for $y\in Y_3$, let
$\hy = f_2\circ f_3\Inv(y)$ and for $y\in Y_1$, 
$\hy = y$. Note that in all cases, $g(y) = g(\hy)$ and that
$\barf\circ g(y) = \{y, \hy\}$ and $g\circ \barf(x) = \{x\}$.

Assume now that $y_n\raw y \in Y$ and we show $g(y_n)\raw g(y) \in X$.
Since $X$ is compact, there is a subsequence and an $x_0\in X$ with 
$g(y_{n_i})\raw x_0$. Thus $\barf(x_0) \supset \LP\{\barf\circ g(y_{n_i})\}
= \LP\{y_{n_i}, \hy_{n_i}\}$ which contains $y$. Thus
$y\in \barf(x_0)$ and so $g(y)\in g\circ \barf(x_0) = \{x_0\}$ as needed.

\end{proof} 

\section{From the infimax minimal set to the IPSA}\label{mainmapsect}
For $\un\in \what$, recall that its asymptotic period
two points are $\ut = \lseq{\beta}.\rseq{\alpha}$ and
$\uht =\lseq{\hbeta}.\rseq{\alpha}$ defined in \eqref{alphabeta}.
\begin{defin}
The infimax minimal set corresponding to a collection
of substitutions $\un\in\what$ is $\Dun = \Cl(o(\ut, S))\subset \Sigma_3$.
\end{defin}
 It follows from Remark 23(b) in \cite{lexi} that $\ut$ is almost periodic
and thus $\Dun$ is, in fact, a
minimal set. Further, by  Theorem 16 in \cite{lexi} it
is aperiodic.  It is standard that a compact minimal
set is equal to the closure of the forward
or backward orbits of any of its elements and so we have:
\begin{lem}\label{minset}
For all $\un\in \what$, $\Dun$ is an aperiodic, minimal set and 
$\Dun = \Cl(o(\uht, S)) = \Cl(o^+(\ut, S)) = \Cl(o^-(\ut, S))$.
\end{lem}

Recall from the introduction that
the strategy here is to produce a geometric representation of the
infimax minimal set $\Dun$ as the attractor of an interval map using the path 
space $\cP$ as an intermediary.
By Theorem~\ref{assignment}(b) the word map $\cW$ gives an injection
from paths to sequences. However, for the special paths in the
$\spec_i$, the image path is a one-sided sequence. There is a natural
way to make these image sequences two-sided as follows. By
Theorem~\ref{assignment}(b), the image under $\cW$ of a path in $\spec_1$  
is a one-sided sequence of the form $S^j(\lseq{\beta}.\epsilon)$
with $j<0$ and so 
we redefine the word map image of paths in $\spec_1$ as  
$S^j(\ut)$ for the same $j$. Similarly, we redefine 
 the word map image of paths in $\spec_{\hone}$ as  
$S^j(\uht)$ for $j<0$. These are natural extensions because 
sequences of the form $S^j(\ut)$ and   $S^j(\uht)$ are not in the
image of $\cW$  and they clearly are in $\Dun$.
In addition as we will prove below, the only sequences of the form
$\lseq{\beta}.\ast$ and $\lseq{\hbeta}.\ast$ in $\Dun$
are $\ut$ and $\uht$ respectively.

The complication comes with extending the word map images of paths in 
$\spec_2$. These images are of the form 
$S^j(\epsilon.\rseq{\alpha})$ for $j\geq 0$. These  one-sided sequences
 can be extended
as both $S^j(\ut)$ and   $S^j(\uht)$, so the natural extension of $\cW$
is not a single valued function. Fortunately, what we
require for  the geometric representation is a map going
the other way, $\Dun\raw \cP$, and this map can be constructed in the
next theorem using Lemma~\ref{toplem} which was, of course, designed
for just this purpose.

\begin{thm}\label{mainmap}
Let $Z=\{S^j(\ut)\colon j\geq 0\} \cup \{S^j(\uht)\colon j\geq 0\}$. 
There exists a continuous surjection $G:\Dun\raw \cP$ which is
an injection on $\Dun\setminus Z$ and two-to-one on $Z$ 
with $G(S^j(\ut)) = G(S^j(\uht))$ for all $j\geq 0$.
\end{thm}

\begin{proof} We start by constructing the various spaces and
maps needed for the application of Lemma~\ref{toplem}. 
To distinguish the specific constructions from
that general lemma we use uppercase letters to denote
the corresponding maps.
Let $X_1 = \cN \sqcup \spec_1\sqcup\spec_\hone$,  
$X_2 =\spec_2$, and $Y=\Dun$.
The definition of $F_1:X_1\raw Y$ depends on
the location of $\Gamma$. When
$\Gamma\in\cN$, let $F_1(\Gamma) = \cW(\Gamma)$.
  When $\Gamma\in\spec_1$, by Theorem~\ref{assignment} we have
$\cW(\Gamma) = S^{j}(\lseq{\beta}.\epsilon)$ for a unique $j<0$
and using this $j$ define 
$F_1(\Gamma) = S^{j}(\lseq{\beta}.\rseq{\alpha})$.
Similarly, when $\Gamma\in\spec_\hone$,
$\cW(\Gamma) = S^{j}(\lseq{\hbeta}.\epsilon)$ for a unique $j<0$
and  define $F_1(\Gamma) = S^{j}(\lseq{\hbeta}.\rseq{\alpha}) 
= S^j(\uht)$.
When $\Gamma\in X_2 =\spec_2$, again by Lemma~\ref{assignment} we have
$\cW(\Gamma) = S^{j}(\epsilon.\rseq{\alpha})$ for a unique $j\geq 0$
and using this $j$ define $F_2(\Gamma) = 
S^{j}(\lseq{\beta}.\rseq{\alpha}) = S^j(\ut)$
and $F_3(\Gamma) = S^{j}(\lseq{\hbeta}.\rseq{\alpha})= S^j(\uht)$.

We thus have as subsets of  $Y = \Dun$,
$
F_1(\spec_1) = \{S^j(\ut)\colon j<0 \},
F_1(\spec_\hone) = \{S^j(\uht)\colon j<0 \},
F_2(\spec_2) = \{S^j(\ut)\colon j\geq 0 \}$,
 and $F_3(\spec_2) = \{S^j(\uht)\colon j\geq 0 \}$.
Let $Y_1 = F_1(\cN) \sqcup F_1(\spec_1)\sqcup F_1(\spec_\hone)$,
$Y_2 = F_2(\spec_2)$, and   
$Y_3 = F_3(\spec_2)$.   

We now show that these spaces and maps have the required properties.
Since $F_1(\spec_1) = o^-(\ut,S)\setminus \{\ut\}$,  
Lemma~\ref{minset} shows that $F_1(\spec_1)$ is dense
in $\Dun$ and thus so is $F_1(X_1)$. 
All the $F_i$ are injective by Theorem~\ref{assignment} 
and Lemma~\ref{stronginjective}. That lemma also
implies that $Y_1\cap(Y_2\cup Y_3) = \emptyset$.
Now if $\Gamma, \Gamma'\in \spec_2$ with $\Gamma\not=\Gamma'$,
by Lemma~\ref{stronginjective} again, $F_2(\Gamma) \not= F_3(\Gamma')$.
Finally, for $\Gamma\in\spec_2$, $F_2(\Gamma) \not= F_3(\Gamma)$
since $\lseq{\beta} = \dots 1.$ and  $\lseq{\hbeta} = \dots 2.$.
Thus all the $Y_i$ are disjoint.

We now check the required limit assertions. Two initial 
 observations about paths in  the IPSA will be needed. The first
is that if $u\not =\epsilon$ is a prefix in an 
edge label, then  $u = 3 1^\ell$ with
$\ell \geq 0$. Thus for all $k>2$ there exists a word $W$
with $\Lambda^{(k)}(u) = W \Lambda^{(\sigma)} (1)$ 
and $\sigma = k$ or $k+1$.  The second
is that if $v\not =\epsilon$ is a suffix in an 
edge label, then  $v =  1^\ell$ with
$\ell > 0$. Thus for all $k>3$ there exists a word $W$
with $\Lambda^{(k)}(v) = \Lambda^{(k-2)} (3) W$. 

Assume that $\Gammam$ is a sequence in $\cP$ with 
$\Gammam\raw\Gamma\in \spec_2$ with 
$$\Gamma = \Gamma_0, \dots, \Gamma_w,
 \prod_{i=w+1}^\infty (\epsilon, 3, 1^{n_i}),$$
and $\Gamma_w\not = (\epsilon, 3, 1^{n_{w+1}})$. 
By Lemma~\ref{specprop}(c) and Lemma~\ref{form}(b),
 we know that  for all $k>w$ there
is a  $j\geq 0$ (independent of $k$) so that
 $$\cW(\Gamma) = 
S^j(\epsilon.\rseq{\alpha}) = 
\epsilon S^j\Lambda^{(k)}(\epsilon.3) \Lambda^{(k)}(v_k)\dots
$$

Now if $\Gamma_m\in\spec_2$ and $d(\Gammam, \Gamma) < 1/(w+2)$,
then $\Gamma_m = \Gamma$, so we may assume that  
$\Gamma_m\not\in\spec_2$ for all sufficiently large $m$.
Let $M(m)$ be such that $(\Gammam)_i = \Gamma_i$ for 
$i=1, \dots, M(m)-1$ and $(\Gammam)_{M(m)} \not= \Gamma_{M(m)}$.
If $M(m) > w$ this implies that $(\Gammam)_{M(m)} =
 (\epsilon, 3, 1^{n_{M(m)+1}})$,
$(\Gammam)_{M(m)+1} = (\epsilon, 2, \epsilon)$, and
$(u_{(m)})_{M(m)+2} =  3 1^\ell$.
Thus using the first observation above, 
$$
\cW(\Gammam) = \dots \Lambda^{(\sigma(m))}(1)
S^j\Lambda^{(M(m))}(\epsilon.3) \Lambda^{(M(m))}(v_{M(m)}) \dots
$$
with $\sigma(m) = M(m)+1$ or $M(m)+2$. Now of necessity $M(m)\raw\infty$ 
as $m\raw\infty$ and so $\sigma(m)\raw\infty$ also and thus
 by \eqref{alphabeta}, 
$\LP\{\Gammam\} \subset \{ S^j(\lseq{\beta}.\rseq{\alpha}) \cup
S^j(\lseq{\hbeta}.\rseq{\alpha})\}$, as required.

Now assume that $\Gammam$ is a sequence in $\cP$ with 
$\Gammam\raw\Gamma\in \spec_1$ with 
$$\Gamma = \Gamma_0, \dots, \Gamma_w,
 \prod_{i=1}^\infty \left((\epsilon, 2, \epsilon), 
(3 1^{n_{w + 2i + 1}}, 1,\epsilon)\right)$$
with $\Gamma_w\not = ( 3 1^{n_{w+1}}, 1, \epsilon)$
and $w$ is even. 
By Lemma~\ref{specprop} and Remark~\ref{negj},
 we know that  for all $2k>w$ there
is a  $j< 0$ (independent of $k$) so that
 $$\cW(\Gamma) = 
S^j(\lseq{\beta}.\epsilon) = \dots \Lambda^{(2k)}(u_{2k})
 S^j\Lambda^{(2k)}(1.\epsilon) \epsilon\dots
$$

As in the previous argument, we may assume that  
$\Gamma_m\not\in\spec_1$ for all sufficiently large $m$.
Let $M(m)$ be such that $(\Gammam)_i = \Gamma_i$ for 
$i=1, \dots, M(m)-1$ and $(\Gammam)_{M(m)} \not= \Gamma_{M(m)}$.
Since there is a unique outgoing edge from state $2$ 
this implies that $M(m)$ is even. Examining the IPSA to determine 
 $\Gamma_{M(m)}$ and its successors and then 
 using the second observation above, 
$$
\cW(\Gammam) = \dots  \Lambda^{(M(m))}(u_{M(m)})
S^j\Lambda^{(M(m))}(1.\epsilon)\epsilon \Lambda^{(\sigma)}(3) \dots
$$
with $\sigma(m) = M(m)-1$ or $M(m)-2$. Now   $M(m)$ and $\sigma(m)\raw\infty$ 
as $m\raw\infty$  and thus
 by \eqref{alphabeta}, 
$\Gammam \raw  S^j(\lseq{\beta}.\rseq{\alpha})$, as required.
The case $\Gammam\raw\Gamma\in \spec_\hone$ is similar.

Thus we have verified all the required properties to utilize
Lemma~\ref{toplem} and so 
$G:\Dun\raw \cP$  defined by
$G(y) = F_i\Inv(y)$ when $y\in Y_i$ is continuous,
injective on $X_1$ and two-to-one as indicated on $Z = Y_2\cup Y_3$. 
\end{proof}

\section{The one-sided case}
Theorem~\ref{mainmap} describes the relation of the infimax
minimal set $\Dun$ and the path  space $\cP$ using the map $G$ 
which is, roughly speaking, the inverse of the extended
word map. Theorem~\ref{mainmap} shows the only non-injectivity of $G$ is in 
sending points on the forward orbits of $\ut$ and $\uht$ to
the same path. Recall now that  
$\ut = \lseq{\beta}.\rseq{\alpha}$ and 
$\uht = \lseq{\hbeta}.\rseq{\alpha}$ and thus for $j\geq 0$,
$\pi(S^j(\ut)) = \pi(S^j(\uht))$ where $\pi:\Sigma_3\raw \Sigma_3^+$
is defined as $\pi(\us) = s_0 s_1 \dots$. This suggests 
that the path space $\cP$ is, in fact, homeomorphic
to the infimax minimal set restricted
to right one-sided sequences defined as 
 $\Dunp = \Cl(o^+(\rseq{\alpha},S))$. This is the main content
of next theorem.

The required homeomorphism is defined using 
 notation as in the proof of Theorem~\ref{mainmap},
by $H\colon\cP\raw \Dunp$ when $\Gamma\in X_1$
 as $H(\Gamma) = \pi\circ F_1(\Gamma)$
 and for $\Gamma\in X_2$, 
$H(\Gamma) = \pi\circ F_2(\Gamma)$. 
 Note that when $\Gamma\in X_2$,
$\pi\circ F_3(\Gamma) = \pi\circ F_2(\Gamma) = H(\Gamma)$.
Thus informally, $H = \pi\circ \barF$.
Note that it is  standard that $\Dunp = \pi(\Dun)$.

\begin{thm}\label{onesided} $\ $
\begin{enumerate}[(a)]
\item The map $H:\cP\raw \Dunp$ is a homeomorphism.
\item The map $\pi:\Delta_{\un} \raw \Delta^+_{\un}$ is  
 injective except for the case where 
 $\us = S^j(\ut)$ and  $\uhs = S^j(\uht)$
for the same $j\geq 0$, in which case $\pi(\us) = \pi(\uhs)$.
\item Each $\rseq{s}\in \Dunp$ has a unique inverse under the shift
$S$ with the sole exception that $S^{-1}(\rseq{\alpha}) = 
\{ \rseq{1\alpha}, \rseq{2\alpha} \}$.
\end{enumerate}
\end{thm}

\begin{proof}
Since $\pi$ is continuous, the limit assertions proved
in Theorem~\ref{mainmap} show that $H$ is continuous. 
Since $\cP$ is compact, it suffices to show $H$ is injective.
Using Lemma~\ref{stronginjective}(b), $\Gamma\not=\barGamma$
implies $H(\Gamma)\not=H(\barGamma)$ except for the possibility
of a pair $\Gamma$, $\barGamma$ as in \eqref{badpair1} or
\eqref{badpair2}. To
finish part (a) we assume that there exists such a pair
with  $H(\Gamma)=H(\barGamma)$ and obtain a contradiction.
In the remainder of the proof we will use the notation
of Theorem~\ref{mainmap}.

First note that from its definition in \eqref{badpair1} or
\eqref{badpair2}, $\barGamma 
\in\spec_1\cup\spec_\hone$ and so via Remark~\ref{negj} and the
definition of $F_1$,
\begin{equation*}
 F_1(\barGamma) = \ldots S^{-j}(\Lambda^{(w)}(\epsilon.a_w))\rseq{\alpha}
\end{equation*}
for some $j>0$. Again via Remark~\ref{negj} for that same $j$, 
\begin{equation*}
F_k(\Gamma) = \ldots S^{-j}(\Lambda^{(w)}(\epsilon.a_w))\rseq{\sigma}
\end{equation*}
for some one-sided sequence $\rseq{\sigma}$, where
$k = 1$ or $2$ depending whether $\Gamma$ is in $X_1$ or
$X_2$.
Thus if $H(\Gamma)=H(\barGamma)$, left shifting by $j>0$ we have
$\rseq{\alpha} = \rseq{\sigma}$.

Now $F_1(\barGamma)$ and $F_k(\Gamma)$ are both contained in
$\Dun$ which is shift invariant, and thus 
$S^j(F_1(\barGamma))$ and $S^j(F_k(\Gamma))$ are also
both contained in $\Dun$. Thus by Theorem~\ref{mainmap} there
are $\barGamma', \Gamma'\in\cP$ with 
$S^j(F_1(\barGamma))\in \barF(\barGamma')$ and 
$S^j(F_k(\Gamma))\in \barF(\Gamma')$.
Thus $H(\Gamma') = H(\barGamma') = \rseq{\alpha}$. 
Now it follows from Lemma~\ref{specprop}(c)
 that $H(\Gamma_\alpha) = \rseq{\alpha}$,
and also note that since $\Gamma_\alpha  
= \prod_{i=1}^\infty (\epsilon, 3, 1^{n_i})$, it could not be in 
a pair of type \eqref{badpair1} or \eqref{badpair2}.  
Thus Lemma~\ref{stronginjective}(b)
implies that $\Gamma' = \barGamma' = \Gamma_\alpha$ in which
case $\cW(\Gamma)= \cW(\barGamma)$ and so  
by Lemma~\ref{stronginjective}(a), $\Gamma =  \barGamma$, 
a contradiction.

Part (b) follows from part (a), since if $\pi$ is not injective
outside of the special case given, then using Theorem~\ref{mainmap},
 $H$ could not be injective as was just proved. Part (c) follows directly from
part (b).
\end{proof}

\begin{rem}\label{onesidedrem}
\begin{enumerate}[(a)] $\ $
\item 
Theorem~\ref{onesided} says that a given right infinite
sequence $(\us)_R$ with $\us\in\Dun$ has a unique
left extension yielding a full sequence  in $\Dun$ 
except when the right infinite sequence is in
 $o^+(\ut) = o^+(\uht)$. In dynamical language, this
says that the semiconjugacy $\pi:(\Dun,S) \raw (\Dunp,S)$
is injective except on $o^+(\ut)$ and $o^+(\uht)$.
\item
 Theorem~\ref{onesided} also says the noninjectivity
of $G$ is the same as that of $\pi$, or more precisely,
$G = H^{-1}\circ \pi$.
\end{enumerate}
\end{rem}

\section{Asymptotic stable projection}
\subsection{Preliminaries}
As noted in the introduction, the first step in Rauzy's method of
 geometric representation of a substitution minimal
set is to find the stable direction of the Abelianization.
For the analog in the S-adic case we need  an
 asymptotic stable direction of the composition
of an arbitrary sequence of the Abelianization
matrices $A_n$ defined in \eqref{subabel}. The standard method for  this type
of proof uses the Hilbert metric and imitates Garret Birkhoff's
proof of the \PF\ theorem. We just give the standard definitions and
results for the case of interest here, but they of
course hold in a more general context.

The open octants in $\R^3$ are cones specified by 
a triple $(\delta_1, \delta_2, \delta_3)$ with each 
$\delta$ being $+$ or $-$ and
$\cC_{\delta_1, \delta_2, \delta_3}
 = \{\vv\in\R^3\colon \delta_1 v_1>0, \delta_2 v_2>0, \delta_3 v_3>0\}$.
Thus, for example, the positive cone is $\cC_{+, +, +}$. The projectivations
of these cones is given by their intersection with the unit sphere,
$S^3_{\delta_1, \delta_2, \delta_3}
 = \cC_{\delta_1, \delta_2, \delta_3} \cap S^3 $ 

The Hilbert projective pseudometric on  $\cC_{+, +, +}$ is
:
\[ \rho(\alpha, \beta) = 
\log \max_{1 \le i,j \le k} \frac{\alpha_i \beta_j}{\alpha_j \beta_i}. \]
This pseudometric restricts to a metric on  $S^3_{+, +, +}$ which
is equivalent to the standard metric.

If $A$ is a strictly positive $3 \times 3$ matrix and
 $f_A:S^3_{+, +, +}\raw S^3_{+, +, +}$ is defined as 
 $f_A (\alpha) = \frac{A\alpha}{\|A\alpha\|_2}$, then 
Birkhoff's Contraction Theorem (\cite{birkhoff}, \cite{Carroll})
 says that
\[\rho\left( f_A(\alpha), f_A(\beta)\right) 
\le \tau(d(A)) \rho(\alpha, \beta),\]
 where
\[ d(A) = \max_{1 \le i, j, k, l\le 3}
 \frac{a_{ik} a_{jl}}{a_{il}a_{jk}} \ge 1.\] 
and
\[ \tau(d) = \frac{\sqrt{d} - 1}{\sqrt{d} + 1}.\]

\subsection{The case of a single matrix}
To clarify the argument for the existence of an asymptotic stable projection
of the infinite composition of matrices, we first consider 
the case of a single matrix $A_n$ for some $n>0$.

The matrix $A_n$ has three distinct eigenvalues 
$0 < \lambda_1 < 1 < -\lambda_2 < \lambda_3$ (proof of
Lemma 52 in \cite{beta}). Let the  
corresponding left and right eigenvectors be 
$ \vu_1, \vu_2, \vu_3$ and $\vv_1, \vv_2, \vv_3$.
Since the eigenvalues are all distinct, we may choose the eigenvectors so  that
$\vu^{(i)} \cdot \vv^{(j)} = \delta_{ij}$, the Kronecker-delta. In particular,
if $V = \sppan(\vv^{(2)}, \vv^{(3)})$ (the unstable subspace), then 
$\vu_1 \perp V$. The projection down the unstable subspace onto the
stable subspace  is given by 
 $\Phi(\vw) = \vu_1 \cdot \vw$. Thus if $\vw$ is written in
right eigen-coordinates, $\vw = \sum c_i \vv_1$, then $c_1 = \Phi(\vw)$.
Note that $\Phi$ is the stable eigen-covector of $A_n$ in that 
$\Phi(A_n\vw) = \lambda_1 \Phi(\vw)$. Thus to accomplish a stable projection
in an asymptotic Rauzy fractal construction we need the analog of the
stable eigen-covector.

To find this covector in the infinite product case, we use
the analog of Birkhoff's proof of the \PF\ theorem which finds the 
strongest unstable eigenvector. However since we require
 a stable covector,   we seek the unstable left eigenvector
of $A_n^{-1}$ or the unstable right eigenvector of 
$A_n^{-T}$. Conjugating $A_n^{-T}$ by the involution
$\tau$ which sends $\be_3\raw - \be_3$ and leaves $\be_1$ and
$\be_2$ fixed and we obtain a nonnegative matrix
\begin{equation*}
 B_n := \tau A_n^{-T} \tau = 
\begin{pmatrix}
0 & 1 & 1\\
1 & 0 & 0 \\
0 & n & n+1
\end{pmatrix}.
\end{equation*}
A simple computation yields that
 $(B_n)^3 > 0$, and so using the Birkhoff's contraction Theorem, 
 $$\frac{(B_n)^k (\vw)}{\|(B_n)^k (\vw)\|} \raw\tau(\vu_1)^T$$
 for any $\vw\in\cC_{+, +, +}$, where $\|\cdot\|$ is the two norm. 
Translating, 
$$
\frac{(A_n^{-T})^k (\vw)}{\|(A_n^{-T})^k (\vw)\|} \raw  \vell'
$$
for any $\vw\in\cC_{+, +, -}$, where $(\vell')^T$ is the unit left eigenvector
of $A_n$ for eigenvalue $\lambda_1$ with $\vell'\in\cC_{+, +, -}$. For technical
reasons that will be clear later, it will be easier to work with
the unit left eigenvector
of $A_n$ for eigenvalue $\lambda_1$ with $\vell\in\cC_{-, -, +}$ and we then
have 
$$
\frac{(A_n^{-T})^k (\vw)}{\|(A_n^{-T})^k (\vw)\|} \raw  \vell
$$
for any $\vw\in\cC_{-, -, +}$,
\begin{rem}
A part of the process just described for the Abelianizations
has a nonlinear analog.
Treating the substitution $\Lambda_n$ as an  homomorphism  
of the free group generated by 
$\{\delta_1,\delta_2,\delta_3\}$ it maps $\delta_1\mapsto\delta_2,
\delta_2\mapsto \delta_3\delta_1^{n+1},$ and
 $\delta_3\mapsto \delta_3\delta_1^n$. This is 
 an automorphism with inverse $\delta_1\mapsto \delta_3^{-1} \delta_2,
\delta_ 2\mapsto \delta_1,$ and $
\delta_3\mapsto \delta_3(\delta_2^{-1}\delta_3)^n$.  
 Conjugating by
 the free group involution $\delta_1\mapsto\delta_1,
 \delta_2\mapsto\delta_2, $ and $\delta_3\mapsto\delta_3^{-1}$
yields $\delta_1\mapsto \delta_3\delta_2, 
\delta_2\mapsto \delta_1$, and $
\delta_3\mapsto (\delta_3\delta_2)^n \delta_3$
which may be viewed as a substitution. Because
$\Lambda_n$ is inverse-Pisot, this new ``inverse'' substitution
is Pisot.
There are a great many results and constructions associated
with Pisot substitutions (see \cite{ABB} for a thorough case study and
\cite{BST} for results on Pisot S-adic families).
A natural question is whether there is any kind of ``duality'' so
that the application of these methods
to the infimax ``inverse'' substitutions can yield  insights
into the infimax family itself.
\end{rem}

\subsection{The infinite composition}
A family of matrices $\{C_n\}$ is called \textit{eventually positive with
constant $N$} if  every product of $N$ matrices from
the family is strictly positive,
$C_{n_1} C_{n_2} \dots C_{n_N} > 0$. Note that there are
many similar notions in the literature under a variety
of names most incorporating ``primitive'', ``positive'' or
``Perron-Fr\"obenius'' in some manner.
Given a list of indices $ (n_1, n_2, \dots) = \un\in \what$ and 
family of $3\times 3$ matrices $C_n$, the product of the first
$k$ matrices is written
$C^{(k)} =  C_{n_1} C_{n_2} \dots C_{n_k}$.

The following is a standard result used in many areas
of mathematics.
It is an easy consequence of Birkhoff's contraction theorem
and the well-known fact that if $A$ is 
a non-negative matrix which has no zero row, its induced map 
$f_A:S^3_{+, +, +} \raw S^3_{+, +, +}$ 
 does not increase
the Hilbert metric, or $\rho(f_A( z_1), f_A(z_2)) \leq \rho(z_1, z_2)$,
for all $z_1, z_2\in S^3_{+, +, +}$.
\begin{lem}\label{Birkhoff}
 Let $C_n\geq 0$  be an eventually positive  family of
 non-negative matrices with constant $N$ and each
$C_n$  has no zero row. Further, assume 
 there is a  $\kappa < 1$
 such that
$\tau(C_{n_1} C_{n_2} \dots C_{n_N}) < \kappa$ for all 
$n$-tuples $(n_1, \dots, n_N)$.
Then for all $\un\in \what$ there exist a $z_{\un} \in S^3_{+, +, +}$ 
such that for all $\vw\in\cC_{+,+,+}$, 
\begin{equation*}
\lim_{k\raw\infty}\frac{C^{(k)}(\vw)}{\|C^{(k)}(\vw)\|_2} \raw  z_{\un}.
\end{equation*}
In addition, if $\vv^{(k)}$ is the unit length positive right
\PF\ eigenvector of the finite product
$C^{(k)}$ with $k>N$, then $\vv^{(k)}\raw z_{\un}$.
\end{lem}
Using this lemma we get the existence of the asymptotic
stable covector.
\begin{thm}\label{mainconv}
For all sequences $\un\in \what$, there exists an
 $\vell_{\un}\in S^3_{-, -, +}$ such that for all
$\vw\in \cC_{-, -, +}$,  
\begin{equation*}
\lim_{k\raw\infty}\frac{(A^{(k)})^{-T}(\vw)}{\|(A^{(k)})^{-T}(\vw)\|} 
\raw  \vell_{\un}.
\end{equation*}
In addition, if $\Phi^{(k)}$ is the stable eigen-covector of
the finite product $A^{(k)}$, 
then $\Phi^{(k)}\raw \Phi_\un$, where $\Phi_\un$ is defined
as $\Phi_{\un}(\vw) = \vell_{\un} \cdot \vw$. 
\end{thm}

\begin{proof}
First note that 
\begin{equation*}
D := B_c B_b B_a = 
\begin{pmatrix}
b & 1 + a(b+1) & 1 + (a+1)(b+1)\\
1 & a & a+1\\
b(c+1) & c + a(b+1)(c+1) & c + (a+1)(b+1)(c+1)
\end{pmatrix}
\end{equation*}
and so the family $\{B_n\}$ is eventually positive and also note that 
each $B_n$ has no zero row. Thus to use Lemma~\ref{Birkhoff}
we must find a uniform $\kappa < 1 $ with 
$\tau(D) < \kappa$ for all $a,b,c$.

Examining the formula for $D$, we see that each
term in the definition of $d(D)$, 
\begin{equation}\label{quot}
 \frac{D_{ik} D_{jl}}{D_{il}D_{jk}},
\end{equation} has the property that the term $a^{m_1} b^{m_2} c^{m_3}$ with
$m_1 + m_2 + m_3$ maximal occurs for  a unique triple 
$(m_1, m_2, m_3)$ in both the numerator and denominator, and 
 these maximal
terms are the same in the numerator and denominator. In addition,
every other term in the numerator and denominator,
 $a^{m_1'} b^{m_2'} c^{m_3'}$, has the property that
$m_i'\leq m_i$ for all $i$. Dividing the numerator and denominator
by the maximal terms shows that the quotient in \eqref{quot} is bounded
above by the maximum number of terms $a^{m_1'} b^{m_2'} c^{m_3'}$
in any numerator or denominator, which is $9$ and so for all triples,
$\tau(D) \leq (\sqrt{9}-1)/(\sqrt{9} + 1) = 1/2$.
\end{proof}
\begin{rem}\label{missing} $\ $
\begin{enumerate}[(a)]
\item The theorem does not assert the existence of an 
assymptotic stable eigenvalue or smallest Lyapunov exponent.
Its existence can be obtained from the Multiplicative
Ergodic Theorem for almost every sequence
with respect to a shift invariant measure $\mu$
on $\Sigma_\infty^+$ as long as $B_n$ is
$\mu$-log integrable. 
\item As mentioned in the introduction, the
solution to the infimax problem for all asymptotic
symbol frequency vectors requires the inclusion
of the $n=0$ substitution in the family \eqref{subdef}
(\cite{lexi}). However when it is included
the family is no longer eventually positive and the analysis
requires many special cases and results without 
much significant additional payoff. 
\end{enumerate}
\end{rem}
The next lemma is essential for a subsequent estimate.
\begin{lem}\label{signs}
For all sequences $\un\in \what$ and for all $k>0$,
$(A^{(k)})^T (\vell_{\un}) \in \cC{-, -, +}$.
\end{lem}

\begin{proof} Using Theorem~\ref{mainconv},
\begin{align}
(A^{(k)})^T (\vell_{\un}) &= 
(A^{(k)})^T 
\left(\lim_{j\raw\infty}\frac{(A^{(j)})^{-T}(\vell_{\un})}{\|(A^{(j)})^{-T}(\vell_{\un})\|}\right)\\
&= \lim_{j\raw\infty} 
\frac{\|A_{n_{k+1}}^{-T} \dots A_{n_{j}}^{-T}(\vell_{\un})\|}{\|(A^{(j)})^{-T}(\vell_{\un})\|}
\frac{A_{n_{k+1}}^{-T} \dots A_{n_{j}}^{-T}(\vell_{\un})}{\|A_{n_{k+1}}^{-T} \dots A_{n_{j}}^{-T}(\vell_{\un})\|}\nonumber
\end{align}
Applying Theorem~\ref{mainconv} to the sequence 
$\un' = n_k n_{k+1}\ldots$ 
we have that the far right hand term converges to  some
$\vell_{\un'}\in \cC_{-, -, +}$.
An easy induction shows that the left hand side is a nonzero vector.
 Taking norms of both sides then
yields that  the first term on the right hand side converges to a  
scalar $c > 0$. Thus
 $(A_{(k)})^T (\vell_{\un}) = c\, \vell_{\un'}\in \cC_{-, -, +}$,
as required.
\end{proof}

\section{From paths in the IPSA to the real line}\label{Psidefsect}
For an unpointed, length $k$ word $w$, its Abelianization 
is 
\begin{equation*}
P(w) = \sum_{i=0}^{k-1} \be_{w_i} = (|w|_1, |w|_2, |w_3|) \in \R^3.
\end{equation*}
By convention $P(\epsilon) = 0$.
The Abelianization of the substitution $\Lambda_n$ is
 the matrix $A_n$ in \eqref{subabel} and is
defined by $(A_n)_{i,j} = |\Lambda_n(j)|_i$.  The action of
$A_n$ on Abelianized words satisfies
 $P\circ \Lambda_n = A_n\circ P$.

We now define a number of objects which 
depend on a given a $\un\in \what$, but we usually
suppress the dependence
on $\un$ for notational compactness.
By Theorem~\ref{mainconv}, each $\un$ yields 
an asymptotic stable eigen-covector $\Phi:\R^3\raw\R$
with $\Phi(x) = \vell \cdot x$ and $\vell$ as in the theorem.
Let $\phi = \Phi\circ P$ and so $\phi(w_1 w_2) = 
\phi(w_1) + \phi(w_2)$  and
\begin{equation*}
\phi(\Lambda_n(w)) = \vell\cdot(A_n P(w)) = (A_n^T \vell)\cdot P(w).
\end{equation*}
Letting $\vell^{(k)} = (A^{(k)})^T \vell$, we have for $a=1, 2, 3$, 
\begin{equation*}
\phi(\Lambda^{(k)}(a)) = ((A^{(k)})^T \vell)\cdot \ve_a = \ell^{(k)}_a,
\end{equation*}
the $a^{th}$ component of $\vell^{(k)}$. Note that according to 
Lemma~\ref{signs}, for all $k$, $\ell^{(k)}_1, \ell^{(k)}_2 <0$ and 
$\ell^{(k)}_3 > 0$. 

Motivated by the way a path generates a sequence in
the Dumont-Thomas expansion in
the definition of $\cW$ in \eqref{seqmapdef},
 define $\Psi_R, \Psi_L:\cP\raw \R$ by
\begin{align}
\Psi_L(\Gamma) &= \phi(u_0) + \sum_{j=1}^\infty \phi(\Lambda^{(j)}(u_j))\\
\Psi_R(\Gamma) &= \phi(a_0 v_0) + \sum_{j=1}^\infty \phi(\Lambda^{(j)}(v_j))\nonumber
\end{align}

\begin{rem} 
The infinite sum $\sum \phi(s_i)$ obviously
doesn't converge. One of the many advantages of the
Dumont-Thomas expansion is that it naturally yields
a grouping of terms so that 
the projection onto the stable subspace is expressed as 
a convergent sum. A different but equivalent method was
used by Holton and Zamboni in \cite{HZ1} and \cite{HZ2} 
for the single substitution
case. If  $\uw$ is the fixed point of  substitution,
for any sequence  $\uv\in\Cl(o(\uw, S))$ by definition
 there are $n_i\raw\infty$
with $S^{n_i}(\uw)\raw \uv$.  They  show that
with $\phi$ the composition of Abelianization and projection as
above, the sequence $\phi(w_0 w_1 \dots w_{n_i-1})$ converges to
a $f(\uv)$
yielding a uniformly continuous representation  
$f: \Cl(o(\uw, S))\raw\C$.
\end{rem}

\begin{thm}\label{Psicont}
For each $\un\in \what $,  $\Psi_L$ and
$\Psi_R$ define continuous functions on $\cP$ with
$\Psi \geq 0$ and $\Psi_R = -\Psi_L$.
\end{thm} 
\begin{proof}
By definition $\vell^{(k)} = A^T_{n_k} \vell^{(k-1)}$ and
so using the formula for $A_{n_{k}}^T$ and  Lemma~\ref{signs} we get
\begin{align}\label{zero}
\ell_1^{(k)} &= \ell_2^{(k-1)} < 0\\
\ell_2^{(k)} &= (n_k + 1) \ell_1^{(k-1)} + \ell_3^{(k-1)} < 0\nonumber\\
\ell_3^{(k)} &= n_k \, \ell_1^{(k-1)} + \ell_3^{(k-1)} > 0.\nonumber
\end{align}
Thus 
\begin{equation}\label{one}
- n_k\, \ell_1^{(k-1)} < \ell_3^{(k-1)} < - (n_k + 1) \ell_1^{(k-1)},
\end{equation}
and so
\begin{equation}\label{two}
|\ell_2^{(k-2)}| = |\ell_1^{(k-1)}|
 < \frac{\ell_3^{(k-1)}}{n_k} \leq  \ell_3^{(k-1)}.
\end{equation}
Using the right hand inequality of \eqref{one},
\begin{align}\label{three}
0 &< \frac{\ell_3^{(k)}}{\ell_3^{(k-1)}} = 
\frac{n_k  \ell_1^{(k-1)} + \ell_3^{(k-1)}}{\ell_3^{(k-1)}}\nonumber\\
&= 1 + \frac{n_k  \ell_1^{(k-1)}}{\ell_3^{(k-1)}}
< 1 + \frac{n_k  \ell_1^{(k-1)}}{-(n_k + 1)   \ell_1^{(k-1)}}\\
&\leq 1-\frac{1}{2} = \frac{1}{2},\nonumber
\end{align}
and so when $k>0$, 
\begin{equation}\label{four}
\ell_3^{(k)} < \frac{1}{2^{k}} \ell_3^{(0)}.
\end{equation}

Next we examine the terms making up the sum defining $\Psi_L$.
Using the IPSA, we have that the possibilities for $u_k$ are
\begin{equation*}
u_k = \epsilon, 3, 31, \dots, 3 1^{n_{k + 1}}.
\end{equation*}
Thus when $\phi(\Lambda^{(k)}(u_k))\not = 0$, 
$\phi(\Lambda^{(k)}(u_k)) =  \phi(\Lambda^{(k)}(3 1^m))
= \ell_3^{(k)} + m \ell_1^{(k)}$ with $0\leq m \leq n_{k + 1}$.
Note that 
when $m =n_{k + 1}$, $\ell_3^{(k)} + n_{k + 1} \ell_1^{(k)} =
\ell_3^{(k+1)}$, using \eqref{zero}.
Thus since $\ell_1^{(k)} < 0$,
when $\phi(\Lambda^{(k)}(u_k)) \not= 0$ and
$0 < m \leq n_{k+1}$, we have 
\begin{equation}\label{ineq1}
0 < \ell_3^{(k+1)} \leq \phi(\Lambda^{(k)}(u_k))   < \ell_3^{(k)},
\end{equation}
 and when $m=0$,
\begin{equation}\label{ineq2}
0 < \ell_3^{(k+1)} < \phi(\Lambda^{(k)}(u_k))  \leq \ell_3^{(k)}.
\end{equation}
The convergence, positivity and continuity of $\Psi_L$ then follow
from \eqref{four}.

The proof for  $\Psi_R$ is similar.
The possibilities for $v_k$ are
\begin{equation*}
v_k = \epsilon, 1, 1^2, \dots,  1^{n_{k} + 1}.
\end{equation*}
Thus when $\phi(\Lambda^{(k)}(v_k))\not = 0$, 
$\phi(\Lambda^{(k)}(v_k)) =  \phi(\Lambda^{(k)}(1^m))
=  m \ell_1^{(k)}$ and note that 
when $m =n_{k + 1}$, then
 $ (n_{k} + 1) \ell_1^{(k)} = \ell_2^{(k+1)}-\ell_3^{(k+1)}$, 
using \eqref{zero}.
Thus  either $v_k = \epsilon$ and
then $\phi(\Lambda^{(k)}(v_k)) = 0$, or 
\begin{equation*}
0 < |\ell_1^{(k)}| \leq |\phi(\Lambda^{(k)}(v_k))|
   \leq |\ell_2^{(k+1)}| + |\ell_3^{(k+1)}|,
\end{equation*}
and  convergence and continuity of $\Psi_R$ again follow
from \eqref{four} after using \eqref{two}.

To show that $\Psi_R = -\Psi_L$  note that Lemma~\ref{form}(b)
gives that for all $k>0$,
\begin{align*}
\Psi_L(\Gamma) + \Psi_R(\Gamma) &= \phi(u_0 a_0 v_0) + 
 \sum _{j=1}^k \left(\phi(\Lambda^{(j)}(u_j)) + \phi(\Lambda^{(j)}(v_j))\right)\\
&+\sum_{j=k+1}^\infty \left(\phi(\Lambda^{(j)}(u_j))
 + \phi(\Lambda^{(j)}(v_j))\right)\\ 
&= \phi(\Lambda^{(k)}(a_k)) + 
\sum_{j=k+1}^\infty \left(\phi(\Lambda^{(j)}(u_j))
 + \phi(\Lambda^{(j)}(v_j))\right).
\end{align*} 
But $\phi(\Lambda^{(k)}(a_k)) = \ell_{a_k}^{(k)} \raw 0$
by \eqref{two} and \eqref{four}
 and since the sums defining $\Psi_L$ and $\Psi_R$ converge we have
 $\Psi_L  + \Psi_R = 0$.
\end{proof}

\section{The Rauzy fractal}
We first construct the Rauzy fractal corresponding to
a list $\un\in\what$ in the real line as the image under $\Psi_L$
of the path space $\cP_\un$. In the next section we compose
$\Psi_L$ with the map $G$ from Section~\ref{mainmapsect}
to get the full representation of $\Dun$ into $\R$.

The Rauzy fractal is $\cR = \Psi_L(\cP)$. It has three natural 
subpieces $\cR_1, \cR_2$, and $\cR_3$ defined by
$\cR_a = \Psi_L(\cP[a])$ where
$\cP[a] = \{\Gamma\in\cP\colon a_0 = a\}$,
all the paths that begin at state $a$.
The conjugacy 
in Theorem~\ref{maintext} below implies  that $\cR$ is a Cantor set.

As noted in the introduction, a central issue in the 
process of a geometric representation is the disjointness
of the subpieces of the Ruazy fractal.
The main
result of this section is the convex hulls
of  the subpieces $\cR_a$ in  $\R$ are disjoint except
perhaps at their endpoints. The argument proceeds by specifying
the endpoints of the intervals using particular paths in $\cP_\un$.
This allows us to specify values of the endpoints of the 
intervals in terms of the coordinates of the asymptotic
stable covector $\vell_{\un}$ given by Theorem~\ref{mainconv}. 
Note that some of these paths have
already been defined but are given new names here for clarity of
exposition.
\begin{align*}
\Gamma_{min} = \Gamma_\alpha &=  \prod_{i=1}^\infty (\epsilon, 3, 1^{n_i}),\\
\Gamma_{3max} &= (\epsilon, 3, 1^{n_1+1}) 
 \prod_{i=1}^\infty \left((\epsilon, 2, \epsilon), 
(3 , 1,1^{n_{2i+1}})\right)\\
\Gamma_{2min} = \Gamma_\hbeta &= 
 \prod_{i=1}^\infty \left((\epsilon, 2, \epsilon), 
(3 1^{n_{2i}}, 1,\epsilon)\right)\\
\Gamma_{2max} &=
 \prod_{i=1}^\infty \left((\epsilon, 2, \epsilon), 
(3 , 1,1^{n_{2i}})\right)\\
\Gamma_{1min} = \Gamma_\beta &= (3 1^{n_{1}}, 1,\epsilon),
 \prod_{i=1}^\infty \left((\epsilon, 2, \epsilon), 
(3 1^{n_{2i + 1}}, 1,\epsilon)\right)\\
\Gamma_{max} &=  \prod_{i=1}^\infty 
\left((3,1,  1^{n_{2i-1}}), (\epsilon, 2, \epsilon)\right)
\end{align*}
\begin{thm}\label{inequality}
Each subpiece  $\cR_a$ of the Rauzy fractal $\cR$ 
is contained in an interval $J_a$ and these intervals
intersect in at most one point.  
Specifically, with the paths defined above, 
\begin{equation*} 
\begin{split}
0 &= \Psi_L(\Gamma_{min}) \leq \cR_3 \leq \Psi_L(\Gamma_{3max}) = -\ell_2 
 = \Psi_L(\Gamma_{2min}) \leq \cR_2\\ &\leq \Psi_L(\Gamma_{2max}) = -\ell_1 =
 \Psi_L(\Gamma_{1min}) \leq \cR_1 \leq \Psi_L(\Gamma_{max}) = \ell_3-\ell_2.
\end{split}
\end{equation*}
Further, the paths that achieve the extremes are 
only those given above, or
$\Psi_L\Inv(0) = \Gamma_{min}$, 
$\Psi_L\Inv(\cR_3\cap \cR_2) = \{\Gamma_{3max}, \Gamma_{2min}\}$,
$\Psi_L\Inv(\cR_2\cap \cR_1) = \{\Gamma_{2max}, \Gamma_{1min}\}$,
and $\Psi_L\Inv(\ell_3-\ell_2) = \Gamma_{max}$.
\end{thm}
\begin{proof}
For any prefix  $u_k$ of an arc label, 
$\phi(\Lambda^{(k)}(u_k)\geq 0$ by \eqref{ineq1} and \eqref{ineq2}.
 Since $\Gamma_{min}$ is
the unique path in $\cP$ with all prefixes equal
to $\epsilon$, it uniquely achieves $0 = \min(\cR)$.
As for $\max(\cR)$, note that any arc with a prefix
label not equal to $\epsilon$ must emerge from the
state $1$. Of these, the maximal contribution to
$\Psi_L$ comes from $\phi(\Lambda^{(k)}(3))$ and since
by \eqref{ineq1} and \eqref{ineq2} again,
 $\phi(\Lambda^{(k+1)}(3)) < \phi(\Lambda^{(k)}(3))$,
whatever path achieves $ \max(\cR)$ must begin at 
state $1$. In addition, since a non-$\epsilon$ prefix
occurs at most on every other level, $\max(\cR)$
is achieved uniquely by $\Lambda_{max}$.

Using the logic of the previous paragraph
 $\max(\cR_2)$ will be achieved   by the unique path 
originating at state $2$ that goes to state
$1$ in the least number of steps, namely, 
$\Lambda_{2max}$. Similarly, $\max(\cR_3)$ is uniquely
achieved by  $\Lambda_{3max}$.

To see what achieves  $\min(\cR_1)$ and $\min(\cR_2)$
we use suffix labels and the fact that $\Psi_R = - \Psi_L$
from Theorem~\ref{Psicont}. Now all suffixes are $\epsilon$ or $1^m$ and
any path $\Gamma$ emerging from state $1$ by definition has
$a_0 =1$ and so $\Psi_R(\Gamma) = \ell_1 + \dots$. Since
$\ell_1<0$ by Theorem~\ref{mainconv}, $\min(\cR_1)$ would be uniquely achieved 
by a path with only $\epsilon$ as a suffix, and this
path is $\Gamma_{1min}$. Similarly, $\min(\cR_2)$  is
uniquely achieved by $\Gamma_{2min}$.

We now compute the values of the extrema. Since 
$\Gamma_{1min}$ and $\Gamma_{2min}$ have all suffixes
equal to $\epsilon$, it follows immediately that 
$\Psi_R(\Gamma_{2min}) = \phi(2) =  \ell_2 =-\Psi_L(\Gamma_{2min})$.
 Similarly,
$\Psi_L(\Gamma_{1min}) = -\ell_1$. 
 
On the other hand, computing $\Phi_L(\Gamma_{1min})$ from the definition
yields
\begin{equation*}
\Phi_L(\Lambda_{1min}) = \sum_{i=0}^\infty
 \phi(\Lambda^{(2i)}(3 1^{n_{2i+1}})) = 
\sum_{i=0}^\infty \ell_3^{(2i)} + n_{2i+1}\, \ell_1^{(2i)}
= \sum_{i=0}^\infty \ell_3^{(2i+1)}
\end{equation*}
using \eqref{zero} in the last equality. Thus
$-\ell_1 = \sum_{i=0}^\infty \ell_3^{(2i+1)}$  Similarly,
$-\ell_2 = \Psi_L(\Gamma_{2min}) = 
\sum_{i=1}^\infty \ell_3^{(2i)}$. 

Now observe that also from the definition, $\Phi_L(\Gamma_{3max})
= \sum_{i=1}^\infty \phi(\Lambda^{(2i)}(3))$, so 
 $\Phi_L(\Gamma_{3max}) = \Psi_L(\Gamma_{2min}) = -\ell_2$.
 Similarly, $\Phi_L(\Gamma_{2max}) = \Psi_L(\Gamma_{1min}) = -\ell_1$.
Finally, \newline
$\Phi_L(\Gamma_{max}) = \sum_{i=0}^\infty \phi(\Lambda^{(2i)}(3))
= \sum_{i=0}^\infty \ell^{(2i)}_3 = \ell_3 - \ell_2$.
\end{proof}
\begin{rem}\label{singsub}
For a single substitution $\Lambda_n$, or equivalently
when $\un = n n n \dots$, the vector $\vell$ is the
left eigenvector corresponding to the stable eigenvalue
$\lambda_1$ of $A_n$. Using the eigenvector equation
and the characteristic polynomial we have that
\begin{equation}\label{ellform}
\ell_1 = -\frac{\lambda_1 \, \ell_3}{1-\lambda^2_1}\ \ \ \text{and}
\ \ \ 
\ell_2 = -\frac{\lambda_1^2\, \ell_3}{1-\lambda^2_1}
\end{equation}
Since $\vell^{(i)} = (A^T)^i \vell = \lambda_1^i \vell$, the summation
formulas from the previous proof 
$-\ell_1 = \sum_{i=0}^\infty \ell_3^{(2i+1)}$ and
$-\ell_2 =  
\sum_{i=1}^\infty \ell_3^{(2i)}$ are geometric series 
for \eqref{ellform}.
\end{rem}
\section{From sequences to the real line}\label{upsilondef}
We now construct the desired real valued-map on the symbolic system $\Dun$
by composing the maps constructed in Sections~\ref{mainmapsect} 
and \ref{Psidefsect}.
\begin{defin}
Let $\Upsilon = \Psi_L\circ G:\Dun\raw [0,\infty)$ and 
 $\Upsilon^+ = \Psi_L\circ H^{-1}:\Dunp\raw [0,\infty)$.
\end{defin}
\begin{rem}\label{rem1} $\ $
\begin{enumerate}[(a)]
\item By Remark~\ref{onesidedrem}, $G = H^{-1}\circ \pi$ and so
$\Upsilon = \Psi_L\circ H^{-1}\circ\pi = \Upsilon^+\circ \pi$. 
\item From the definition, $F_1(\Gamma_\beta) = S^{-1}(\ut)$ 
and so using Theorem~\ref{inequality}, 
$\Upsilon(S^{-1}(\ut)) 
= \Psi_L(\Gamma_\beta) = -\ell_1$. Similarly,
$\Upsilon(S^{-1}(\uht)) = -\ell_2$.
Also, $F_2(\Gamma_\alpha) = \ut$ and $F_3(\Gamma_\alpha) = \uht$,
so again using Theorem~\ref{inequality}, 
$\Upsilon(\ut) = \Upsilon(\uht) = 0$.
\end{enumerate}
\end{rem}
The map $\Psi_L$ and thus the map $\Upsilon$ 
is defined via a summation utilizing
a very specific grouping
into words of a sequence $\us\in\Dun$.  The shifted sequence
 $S(\us)$ could depend on a quite different
grouping. The next lemma confirms that the effect
of the shift is as expected.
\begin{lem}\label{shift}
For all $\us\in\Dun$, $\Upsilon(S(\us)) = \Upsilon(\us) + \phi(s_0)$
and for all $\rseq{s}\in\Dunp$, $\Upsilon^+(S(\rseq{s})) = 
\Upsilon^+(\rseq{s}) + \phi(s_0)$.
\end{lem}
\begin{proof}
We first prove the result for $\Upsilon$.
Examining the IPSA, there are no paths with $\cW(\Gamma)_R = \epsilon$
and exactly two paths with $\cW(\Gamma)_R = a \epsilon$ with $a\not=\epsilon$,
namely, $\Gamma_\beta$ and $\Gamma_\hbeta$, with 
$\cW(\Gamma_\beta)_R = 1\epsilon$ and 
$\cW(\Gamma_\hbeta)_R = 2\epsilon$. We consider 
$\Gamma_\beta$, the other case is similar. Letting
$\us = S^{-1}(\ut)$, by Remark~\ref{rem1}(b),    
 $\Upsilon(\us)  = -\ell_1 = -\phi(s_0)$. On the other hand,
 using Remark~\ref{rem1}(b) again 
since $S(\us) = \ut$ and $G(\ut) = \Gamma_\alpha$, we have 
$\Upsilon(S(\us))= 0 = \Upsilon(\us) + \phi(s_0)$,
proving the result in this case.

Now assume $\us\not= S^{-1}(\ut),  S^{-1}(\uht)$ and so
 $\Gamma = G(\us)$ satisfies $\cW(\Gamma)_i \not=
\epsilon$ for $i = 0,1$. Using Lemma~\ref{form}(b), this implies that 
by picking $k$ large enough we can ensure that 
$$\cW(\Gamma) = \dots \Lambda^{(k+1)}(u_{k+1})\, \Lambda^{(k)}(u_k)\,
S^j(\Lambda^{(k)}(\epsilon.a_k))\,  \Lambda^{(k)}(v_k)\, \Lambda^{(k+1)}(v_{k+1})\dots
$$ 
with $j < |\Lambda^{(k)}(a_k)| - 1$. Now by Lemma~\ref{form}(a), there exists
a length $(k+1)$-path $\gamma$ with $\cW(\gamma) = S^{j+1}(\Lambda^{(k)}(a_k))$.
So letting $\Gamma' = \gamma, \Gamma_{k+1}, \Gamma_{k+2}, \dots$
we have 
$$\cW(\Gamma') = \dots \Lambda^{(k+1)}(u_{k+1})\, \Lambda^{(k)}(u_k)\,
S^{j+1}(\Lambda^{(k)}(\epsilon.a_k))\,  \Lambda^{(k)}(v_k)\, \Lambda^{(k+1)}(v_{k+1})\dots
$$
Thus $\cW(\Gamma') = S(\cW(\Gamma))$ and 
$\Psi_L(\Gamma') = \Psi_L(\Gamma) + \phi(s_0)$. But using the definition of
$G$, $G(S(\us)) = \Gamma'$ and so
 $\Upsilon(S(\us)) = \Upsilon(\us) + \phi(s_0)$, as required.

Since by Remark~\ref{rem1}(a), $\Upsilon = \Upsilon^+ \circ \pi$ and
$\Dunp = \pi(\Dun)$, the result for $\Upsilon^+$ also follows.
\end{proof}

\begin{defin} $\ $
\begin{enumerate}[(a)]
\item Let $\uw = \cW(\Gamma_{2max})$ and $\uhw = \cW(\Gamma_{3max})$.

\item Endow $\Dunp$ with the  lexicographic order. 
Recall that a map $h:X\raw Y$ between two linearly ordered
spaces is \textit{weakly order reversing} if $x<y$ implies $h(x) \geq h(y)$
and \textit{strictly order reversing} if $x<y$ implies $h(x) > h(y)$.
\end{enumerate}
\end{defin}

There are two sources of non-injectivity for $\Upsilon =  \Psi_L\circ G$.
The first is a consequence of the noninjectivity of $\Psi_L$ at
the overlaps of the $\cR_a$ 
given in Theorem~\ref{inequality} (cases (2) and (3) in
Theorem~\ref{Upsilonthm} below).
The second  is a consequence of the noninjectivity of
$G$. This noninjectivity
 is the same as that of $\pi$ given in Theorem~\ref{onesided}(b) 
(case (1) in Theorem~\ref{Upsilonthm} below)
 because by Remark~\ref{onesidedrem}(b), 
$G = H^{-1}\circ \pi$ and $H$ is injective. 
\begin{thm}\label{Upsilonthm}$\ $
\begin{enumerate}[(a)]
\item The map $\Upsilon:\Dun \raw \R$ is continuous and  injective 
except for the following cases:
\begin{enumerate}[(1)]
\item $\Upsilon(S^j(\ut)) = \Upsilon(S^j(\uht))$ for $j\geq 0$.
\item  $\Upsilon(S^j(\ut)) = \Upsilon(S^{j+1}(\uw))$ for $j< 0$.
\item  $\Upsilon(S^j(\uht)) = \Upsilon(S^{j+1}(\uhw))$ for $j< 0$.
\end{enumerate}
\item The map $\Upsilon^+:\Dunp \raw \R$ is  
 continuous and weakly
order reversing. It is strictly order reversing 
except for cases (2) and (3) above, or more precisely, 
$\Upsilon^+(\pi\circ S^j(\ut)) = \Upsilon^+(\pi\circ S^{j+1}(\uw))$ for $j< 0$
and 
$\Upsilon^+(\pi\circ S^j(\uht)) = \Upsilon^+(\pi\circ S^{j+1}(\uhw))$ 
for $j< 0$.
\end{enumerate}
\end{thm}
\begin{proof}
The continuity of $\Psi_L$ is proved in Theorem~\ref{Psicont},  
 that of $G$ in Theorem~\ref{mainmap} and that of $H$ in 
Theorem~\ref{onesided}, thus the continuity of
$\Upsilon$ and $\Upsilon^+$  follows. 

We prove the rest of (b) first. Since $\Dunp = \pi(\Dun)$, each
sequence in $\Dunp$ can be expressed as $\pi(\us)$ for some
$\us\in\Dun$. So 
assume that $\pi(\us) \not= \pi(\us')$ with $\us,\us'\in\Dun$ 
 and let $k\geq 0$ be the
smallest index with $s_k \not = s_k'$ and assume $s_k > s_k'$. Thus
 $\Upsilon(S^k(\us)) \in \cR_{s_k}$ and 
$\Upsilon(S^k(\us')) \in \cR_{s_k'}$. Therefore by Lemma~\ref{inequality},
 $\Upsilon(S^k(\us)) \leq  \Upsilon(S^k(\us'))$ with equality only
when  $S^k(\us) = \uw$ and $S^k(\us') = \ut$ or
$S^k(\us) = \uhw$ and $S^k(\us') = \uht$. Now if 
$\Upsilon(S^k(\us)) <  \Upsilon(S^k(\us'))$, using Lemma~\ref{shift}
 this implies that 
$\Upsilon(\us) + \phi(s_0 s_1 \dots s_{k-1}) < 
\Upsilon(\us') + \phi(s_0' s_1' \dots s_{k-1}')$.
But by assumption $s_i = s_i'$ for $i = 0,\dots, k-1$, and
so $\Upsilon(\us) < \Upsilon(\us')$ as required. 

Remark~\ref{rem1}(a) says that 
$\Upsilon =  \Upsilon^+\circ \pi$ and so 
Theorem~\ref{onesided}(b) implies that 
the only noninjectivity 
that $\Upsilon$ possesses in addition to that of $\Upsilon^+$ comes
from case (1), finishing the proof of (a). 
\end{proof}

\begin{rem}\label{values}
Using Remark~\ref{rem1}(b) and Theorem~\ref{Upsilonthm} we have
 $\Upsilon\Inv(0) = \{\ut, \uht\},
\Upsilon\Inv(-\ell_1) = \{S\Inv(\ut), \uw\}$, and
$\Upsilon\Inv(-\ell_2) = \{S\Inv(\uht), \uhw\}$.
\end{rem}

Theorem~\ref{Upsilonthm} gives us detailed information on the
geometric representation maps $\Upsilon$ and $\Upsilon^+$
 defined from the infimax minimal
sets into the reals. The next step is to connect these maps
to the dynamics on their domain and range.

Lemma~\ref{shift} tells us the required dynamics on the image
of $\Upsilon$ in order  to mirror the shift on $\Dun$. Specifically,
if $[a]\subset \Dun$ is the 
cylinder set $\{ \us\in\Dun\colon s_0 = a\}$, then
$\cP[a] = G([a])$ with $G$ defined in Theorem~\ref{mainmap}
and so each component of the Rauzy fractal $\cR_a$ is
$\cR_a = \Upsilon([a])$ for $a=1,2,3$. Thus by Lemma~\ref{shift},
the shift map on $\Dun$ corresponds to translation by $\phi(a) = \ell_a$
on $\cR_a$.

However, recall that the $\cR_a$ are not disjoint, rather by 
Lemma~\ref{inequality}, $\cR_3\cap \cR_2 = \{-\ell_2\}$ and
$\cR_2\cap \cR_1 = \{-\ell_1\}$. Thus to define the translation
we remove the overlap points and let $\cR' = \cR\setminus\{-\ell_2, \ell_1\}$
and $T:\cR'\raw \cR$ is defined via
\begin{equation*}
T(x) = x + \ell_a\ \text{when} \ x\in\cR_a.
\end{equation*}
We also remove the corresponding points from $\Dun$ and 
let $\Delta' = \Dun \setminus (\Upsilon\Inv(-\ell_2) \cup
\Upsilon\Inv(-\ell_1))$ and so $\cR' = \Upsilon(\Delta')$.
Thus   $\Upsilon\circ S = T \circ \Upsilon$ 
restricted to $\Delta'$.

To get a full dynamical conjugacy we have to be able to
iterate this relation which requires excluding the full
orbits of $\Upsilon\Inv(-\ell_2)$ and $\Upsilon\Inv(-\ell_1))$. 
Thus we let $\Delta'' = 
\Dun\setminus(o(\ut, S)\cup o(\uht, S)\cup o(\uw, S)\cup o(\uhw, S))$ and
 $\cR'' = \Upsilon(\Delta'')$. Now note that 
$S(\Delta'') = \Delta''$ and $T(\cR'') = \cR''$ and so 
we have a full conjugacy between $\Delta''$ under the shift $S$
to $\cR''$ under the translations $T$. Finally, using
Theorem~\ref{Upsilonthm}(b) we can see that $\Upsilon^+$ now gives
a strictly order reversing conjugacy from $\pi(\Delta'')$ to
$\cR''$. Thus we have proved:
\begin{cor}\label{cong1} $\ $
\begin{enumerate}[(a)]
\item Restricted to $\Delta'$ we have $\Upsilon\circ S = T \circ \Upsilon$.
\item  $\Upsilon$ restricted to $\Delta''$ gives a
topological conjugacy from $(\Delta'', S)$ to $(\cR'', T)$.
\item $\Upsilon^+$ restricted to $\pi(\Delta'')$ gives a
strictly order reversing 
topological conjugacy from $(\pi(\Delta''), S)$ to $(\cR'', T)$.
\end{enumerate}
\end{cor}
The last stage of the geometric representation process in the next
section embeds $\cR''$ densely in the attractor of
a specific class of interval maps.

\section{The geometric representation as the attractor of an ITM}
We now define the two-parameter family of ITM which can
occur in the geometric representations of the infimax S-adic family.
Define $T_\mot:I\raw I$ with $I = [0,1]$ and  
$0 < \mu_1 < \mu_2 < 1$ by
\begin{align*}
T_{\mu_1, \mu_2}(x) &= x + 1 - \mu_1 \ &\text{for}\ x\in [0, \mu_1)\\
T_{\mu_1, \mu_2}(x) &= x - \mu_1 \ \ \ \ &\text{for}\ x\in [\mu_1, \mu_2)\\
T_{\mu_1, \mu_2}(x) &= x  - \mu_2 \ \ \ &\text{for}\ x\in [\mu_2, 1].
\end{align*}
See Figure 2. This family is topologically
conjugate to the one studied by Bruin and Troubetzkoy in \cite{BT}.
Specifically, using the conjugacy $x\mapsto 1-x$ their map $T_{\alpha, \beta}$ 
is conjugate to $T_\mot$  with $\beta = \mu_1$ and $\alpha = \mu_2$.
 The dynamical object of
interest is the \textit{attractor} 
$$\Omega_{\mu_1, \mu_2} = \bigcap_{i=0}^\infty T^i_\mot(I).$$
\begin{figure}[t]\label{fig2}
\centering
\includegraphics[width=0.5\textwidth]{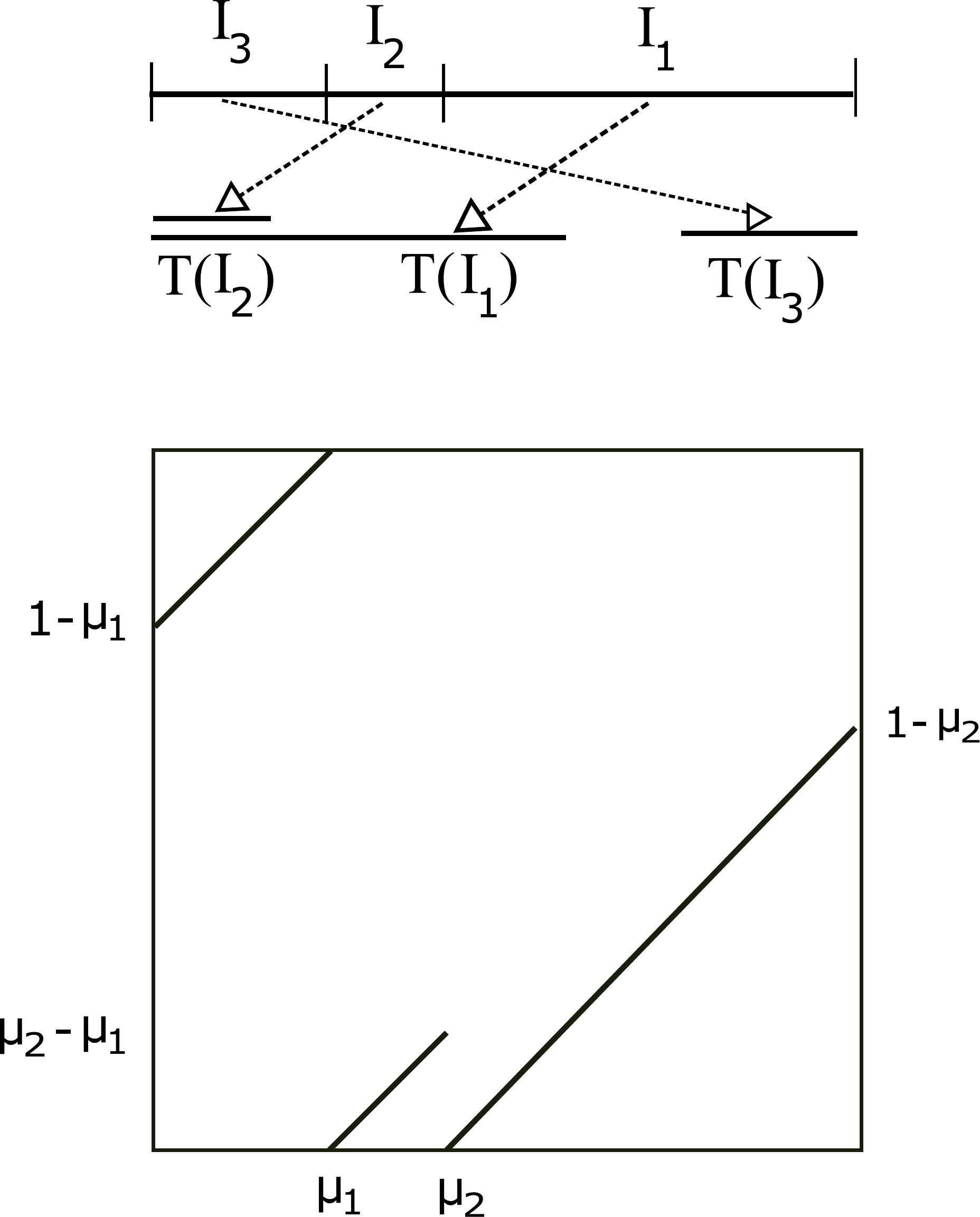}
\caption{Schematic diagram and graph of an ITM in the representing family}
\end{figure}

The maps $T_\mot$ are not continuous, but there is a standard
method to construct a continuous  extension going back to Keane's
work on IET (\cite{Keane}).  The extension is
simplest to describe if we restrict to the case of interest here. So
we assume that $\Tmot$ is such that for all $n>0$,
$\Tmot^n(\mu_1), \Tmot^n(\mu_2)$, and 
$\Tmot^n(0)$ are not equal to either $\mu_1$ nor $\mu_2$
(the importance of zero here is that $0 = 
\lim_{x\raw \mu_1^+} \Tmot(x) = \lim_{x\raw \mu_2^+} \Tmot(x)$).

Informally the  extension is accomplished by
 splitting the backward orbits of 
$\mu_1$ and $\mu_2$ into a pair of orbits. More formally, the
extension is built by extending the linear order
on $I$ to a disjoint union of $I$ and a second copy
of the backward orbits of $\mu_1$ and $\mu_2$.
Let 
$$
\tI = I \sqcup (\bigcup_{j=1,2}\bigcup_{i =0}^\infty \Tmot^{-i}(\mu_j)).
$$
Thus $\tI$ contains two copies of any 
$y\in \cup_{j=1,2}\cup_{i =0}^\infty \Tmot^{-i}(\mu_j)$. Denote the 
version in $I$ itself as $y$ and its copy as $\ty$. Extend the
usual linear order on $I$ to one on $\tI$ by declaring that
$y < \ty$. 

Endowing  $\tI$ with the order topology  makes it a compact metric
space. The map $\Tmot$ naturally extends to a continuous $\tTmot:\tI\raw\tI$
with $\tTmot(\tilde{\mu_1}) =  \tTmot(\tilde{\mu_2}) = 0$ 
and if $y\in I$ has $\Tmot^n(y) = \mu_i$ with $n>0$, then
$\Tmot^n(\ty) = \tilde{\mu_i}$. Define the \textit{extended attractor}
\begin{equation*}
\tOmega_\mot = \bigcap_{i=0}^\infty \tTmot^i(\tI).
\end{equation*}

Given a $\un\in \what$, define $T_\un = \Tmot$ with
$\mu_1 = -\ell_2/(\ell_3 - \ell_2)$ and 
$\mu_2 = -\ell_1/(\ell_3 - \ell_2)$  and let $\tT_\un = \tT_\mot$
with $\vell = (\ell_1, \ell_2, \ell_3)$ as in
Theorem~\ref{mainconv}
\begin{rem}
When $\un = n^\infty$,  Remark~\ref{singsub} yields that the parameters
of the corresponding ITM are $\mu_1 = \lambda_1^2$ and 
$\mu_2 = \lambda_1$ with $\lambda_1$ the stable eigenvalue
of the Abelianization matrix $A_n$ \eqref{subabel}.
\end{rem}
We now rescale the geometric representation maps $\Upsilon$ and
$\Upsilon^+$ so that their image is $I$ and then extend
them so their range is in $\tI$.
The map
$\Upsilon_1^+:\Delta\raw I$ is defined 
as $\Upsilon_1^+ = \Upsilon^+/(\ell_3 - \ell_2)$, while
the extended range map $\tUpsilon^+:\Delta\raw \tI$ is defined
by $\tUpsilon^+ = \Upsilon_1^+$ except for the preimages of
the new points $\ty\in \tI$. To deal with  this case,
 note that using Theorem~\ref{Upsilonthm}
 for $j>0$,
$T_\un^{j-1}\Upsilon^+ \pi S^{-j}(\uht) = T_\un^{j-1}\Upsilon^+ \pi S^{-j+1}(\uhw)
= \mu_1$. So when $y =\Upsilon^+ \pi S^{-j}(\uht)$ let 
$\tUpsilon^+ \pi S^{-j}(\uht) = y$ and $\tUpsilon^+ \pi S^{-j+1}(\uhw) = \ty$.
Similarly,  when $y =\Upsilon^+ \pi S^{-j}(\ut)$ let 
$\tUpsilon^+ \pi S^{-j}(\ut) = y$ and $\tUpsilon^+ \pi S^{-j+1}(\uw) = \ty$.
Finally, let $\tUpsilon = \tUpsilon^+\circ\pi$.

The extension of $\Upsilon$ and its range not only allows us
to extend $T$ to the continuous $\tT$, it also eliminates the noninjectivity
of $\Upsilon$ given in (2) and (3) in Theorem~\ref{Upsilonthm}(a).
Since  the noninjectivity in (1) is eliminated by the passage to
$\tUpsilon^+$ defined on $\Dunp$ we have that $\tUpsilon^+$ is
a homeomorphism onto its image which we show to be $\tOmega_\un$.
Using the commutivity from Corollary~\ref{cong1} we see that
$\tUpsilon^+$ is a topological conjugacy.

There are three intervals in the order topology on $\tI$ that will be used
to code the dynamics on the extended attractor. These are 
\begin{equation}\label{code}
\tI_3 = [0,\mu_1], \tI_2 = [\tilde{\mu}_1, \mu_2],\ \text{and}\ 
\tI_3 = [\tilde{\mu}_2, 1].
\end{equation}
For a point $x\in \tOmega_\un$ its itinerary under $\tT$
with respect to the partition
$\tI_1, \tI_2$ and $\tI_3$ is  $s_0 s_1 \dots$ 
with $s_i = a$ if and only if $\tT^i(x)\in \tI_a$.

The next theorem formalizes the statement in Theorem~\ref{main}
of the introduction: the map $\tUpsilon^+$ provides a conjugacy
of the infimax minimal set $\Dunp$
to the attractor of an extended ITM and further, the partition
just given codes the dynamics exactly as in the symbolic minimal set.

\begin{thm}\label{maintext} $\ $
\begin{enumerate}[(a)]
\item $\tUpsilon_\un^+: (\Dunp, S)\raw (\tOmega_\un, \tT_\un)$ is
a strictly order reversing topological conjugacy and 
for each $\us\in\Dunp $ the itinerary under 
$\tT_\un$ of $\tUpsilon_\un^+(\us)$ with respect to 
the partition $\tI_1, \tI_2$ and $\tI_3$ is $\us$.
\item $\tUpsilon_\un: (\Dun, S)\raw (\tOmega_\un, \tT_\un)$ is
a topological semiconjugacy with its only noninjectivity being
that $\tUpsilon_\un^{-1}(\tT_\un^j(0)) =  \{S^j(\ut), S^j(\uht) \}$
for $j\geq 0$.
\end{enumerate}
\end{thm}
 \begin{proof} By Theorem~\ref{Upsilonthm}(a), $\Upsilon_1^+$ is
continuous and  except for the cases (2) and (3)  given there it
is strictly order reversing.
The extensions $\tUpsilon^+$ and $\tI$ are exactly designed to make
$\tUpsilon^+$ continuous and strictly order reversing for the points in cases
(2) and (3). The commutativity $\tUpsilon^+\circ S = \tT_\mot \circ\tUpsilon^+$
follows from Corollary~\ref{cong1} and the construction of the extensions.
This gives the semiconjugacy of $\Delta_\un^+$
onto the image  $\tUpsilon^+(\Delta_\un^+)$.
 This image is a minimal set
clearly contained in the attractor $\tOmega_\un$.
Now by Proposition 3 in \cite{BT},
 if $\tT_\un$ had a periodic orbit
the attractor would consist of periodic orbits
and we know from Theorem 16 in \cite{lexi}
 that $\Delta_\un^+$ is aperiodic.
Thus $\tT_\un$ is aperiodic and so by Theorem 2.4
in \cite{ST}
it has exactly one minimal set and so 
$\tUpsilon^+(\Delta_\un^+) = \tOmega_\un $. The coding
assertion follows from the conjugacy and the fact as noted
above Corollary~\ref{cong1} that  $\cR_a = \Upsilon([a])$ for
 $a=1,2,3$ where $[a]\subset \Dun$ is the 
cylinder set $\{ \us\in\Dun\colon s_0 = a\}$.  
That completes the proof of part (a). For part (b), since
$\tUpsilon = \tUpsilon^+\circ\pi$, the result for $\tUpsilon$ also
follows using Theorem~\ref{onesided}(b).
\end{proof}
Using the conjugacy we can now transfer what we know about
the dynamics and topology of the infimax minimal set to
the attractor $\tOmega_\un$ of the extended ITM $\tT_\un$.
Using Theorem~\ref{onesided}(b) and the conjugacy we get:
\begin{cor}$\ $
\begin{enumerate}[(a)]
\item Restricted to the extended attractor
 $\tOmega_\un$, the extended self-map $\tT_\un$ is 
injective but for the sole exception that $\tT_\un^{-1}(0) = \{\mu_1,\mu_2\}$.
\item Both $\Omega_\un$ and $\tOmega_\un$ are Cantor sets
\end{enumerate} 
\end{cor}

\section{The infimax family for other $N$}

The analog of \eqref{subdef} with $N$ symbols generates the
solution to the digit  frequency infimax problem for
sequences with elements from $\{1, 2, \dots, N\}$. Bruin shows
using renormalization that
the attractor of an ITM on $N$ intervals is isomorphic
to an S-adic minimal set using \eqref{subdef} with $N$ symbols
 (\cite{B}).
While we focus here on $N=3$, the generalization of the geometric
representation for larger $N$ is fairly
straightforward but heavier in computation and indices.

It has long been known that the $N=2$ infimax problem is 
solved by Sturmian sequences (\cite{glt}, \cite{vee}).
 The $n=2$ infimax family of substitutions, 
$\lambda_n: 1\mapsto 21^{n+1}, 2\mapsto 1^n$ for $n \in\Np$, 
generate these Sturmian sequences for irrational digit frequency
ratios. There are other well known S-adic families that generate
all Sturmians (for a survey, see
\cite{sturmiansurvey}): the $N=2$ family just generates the
infimax Sturmian for each frequency ratio. The geometric
representation in this case differs from $N>2$ in that the
Rauzy fractal is an entire interval and its two subpieces
are intervals with abutting endpoints. The induced map on the 
subpieces switches them and so  the geometric representation
is an IET on two intervals which
can be interpreted as a circle homeomorphism.
The identification of
Sturmian sequences with itineraries of a circle map goes
back to Morse and Hedlund. In addition, when $N=2$ the
substitution induced order on the 
Bratteli-Vershik diagram is proper and so the
Vershik map can be globally defined.

The Sturmian sequences on two symbols have a host 
of nice properties (\cite{sturmiansurvey}) and
there is much literature devoted to their generalization
to more symbols (see, for example, \cite{gbu}). There is not a single
generalization for all the Sturmian properties, but rather
many generalizations, each of which possesses one or a few of
the Sturmians' nice properties. The infimax S-adic family
for $N>2$ generalizes the infimax property of the Sturmians and
their geometric representations, rather than being an
IET on more intervals, is an ITM on $N$ intervals.

\subsection*{Acknowledgements}
Our thanks to the referee for many useful comments.


\normalsize
\baselineskip=17pt

\bibliographystyle{amsplain}
\bibliography{georefs}

\providecommand{\bysame}{\leavevmode\hbox to3em{\hrulefill}\thinspace}
\providecommand{\MR}{\relax\ifhmode\unskip\space\fi MR }
\providecommand{\MRhref}[2]{%
  \href{http://www.ams.org/mathscinet-getitem?mr=#1}{#2}
}
\providecommand{\href}[2]{#2}
\begin{thebibliography}{10}

\bibitem{sturmiansurvey}
P.~Arnoux, \emph{Sturmian sequences}, Substitutions in dynamics, arithmetics
  and combinatorics, Lecture Notes in Math., vol. 1794, Springer, Berlin, 2002,
  pp.~143--198. \MR{1970391}

\bibitem{ABB}
Pierre Arnoux, Julien Bernat, and Xavier Bressaud, \emph{Geometrical models for
  substitutions}, Exp. Math. \textbf{20} (2011), no.~1, 97--127. \MR{2802726}

\bibitem{DT2}
Guy Barat, Val{\'e}rie Berth{\'e}, Pierre Liardet, and J{\"o}rg Thuswaldner,
  \emph{Dynamical directions in numeration}, Ann. Inst. Fourier (Grenoble)
  \textbf{56} (2006), no.~7, 1987--2092, Num{\'e}ration, pavages,
  substitutions. \MR{2290774 (2007k:37011)}

\bibitem{Sadic2}
Valerie Berth{\'e}, \emph{Multidimensional {E}uclidean algorithms, numeration
  and substitutions}, Integers \textbf{11B} (2011), Paper No. A2, 34.
  \MR{3054421}

\bibitem{Sadic}
Val{\'e}rie Berth{\'e} and Vincent Delecroix, \emph{Beyond substitutive
  dynamical systems: {$S$}-adic expansions}, Numeration and substitution 2012,
  RIMS K\^oky\^uroku Bessatsu, B46, Res. Inst. Math. Sci. (RIMS), Kyoto, 2014,
  pp.~81--123. \MR{3330561}

\bibitem{gbu}
Val{\'e}rie Berth{\'e}, S{\'e}bastien Ferenczi, and Luca~Q. Zamboni,
  \emph{Interactions between dynamics, arithmetics and combinatorics: the good,
  the bad, and the ugly}, Algebraic and topological dynamics, Contemp. Math.,
  vol. 385, Amer. Math. Soc., Providence, RI, 2005, pp.~333--364. \MR{2180244
  (2006g:37009)}

\bibitem{ANT}
Val{\'e}rie Berth{\'e} and Michel Rigo (eds.), \emph{Combinatorics, automata
  and number theory}, Encyclopedia of Mathematics and its Applications, vol.
  135, Cambridge University Press, Cambridge, 2010. \MR{2742574 (2011k:68006)}

\bibitem{BST}
Val\'rie Berth\'e, Wolfgang Steiner, and J\"org Thuswaldner, \emph{{Geometry,
  dynamics, and arithmetic of S-adic shifts}}, arXiv:1410.0331 [math.DS], 2014.

\bibitem{birkhoff}
Garrett Birkhoff, \emph{Extensions of {J}entzsch's theorem}, Trans. Amer. Math.
  Soc. \textbf{85} (1957), 219--227. \MR{0087058 (19,296a)}

\bibitem{BK}
Michael Boshernitzan and Isaac Kornfeld, \emph{Interval translation mappings},
  Ergodic Theory Dynam. Systems \textbf{15} (1995), no.~5, 821--832.
  \MR{1356616 (96m:58068)}

\bibitem{lexi}
P.~Boyland, A.~de~Carvalho, and T.~Hall, \emph{Symbol ratio minimax sequences
  in the lexicographic order}, Ergodic Theory and Dynamical Systems \textbf{35}
  (2015), no.~8, 2371--2396.

\bibitem{beta}
\bysame, \emph{On digit frequencies in $\beta$-expansions}, Trans. Amer. Math.
  Soc. \textbf{368} (2016), no.~12, 8633–--8674.

\bibitem{B}
H.~Bruin, \emph{Renormalization in a class of interval translation maps of
  {$d$} branches}, Dyn. Syst. \textbf{22} (2007), no.~1, 11--24. \MR{2308208
  (2008e:37041)}

\bibitem{BT}
H.~Bruin and S.~Troubetzkoy, \emph{The {G}auss map on a class of interval
  translation mappings}, Israel J. Math. \textbf{137} (2003), 125--148.
  \MR{2013352 (2004j:37007)}

\bibitem{CS1}
Vincent Canterini and Anne Siegel, \emph{Automate des pr\'efixes-suffixes
  associ\'e \`a une substitution primitive}, J. Th\'eor. Nombres Bordeaux
  \textbf{13} (2001), no.~2, 353--369. \MR{1879663 (2003b:37023)}

\bibitem{CS2}
\bysame, \emph{Geometric representation of substitutions of {P}isot type},
  Trans. Amer. Math. Soc. \textbf{353} (2001), no.~12, 5121--5144. \MR{1852097
  (2002f:37023)}

\bibitem{Carroll}
Joseph~E. Carroll, \emph{Birkhoff's contraction coefficient}, Linear Algebra
  Appl. \textbf{389} (2004), 227--234. \MR{2080407 (2006a:15042)}

\bibitem{CN}
Julien Cassaigne and Fran{\c{c}}ois Nicolas, \emph{Factor complexity},
  Combinatorics, automata and number theory, Encyclopedia Math. Appl., vol.
  135, Cambridge Univ. Press, Cambridge, 2010, pp.~163--247. \MR{2759107
  (2012e:68288)}

\bibitem{DT}
Jean-Marie Dumont and Alain Thomas, \emph{Systemes de numeration et fonctions
  fractales relatifs aux substitutions}, Theoret. Comput. Sci. \textbf{65}
  (1989), no.~2, 153--169. \MR{1020484 (90m:11022)}

\bibitem{DHS}
F.~Durand, B.~Host, and C.~Skau, \emph{Substitutional dynamical systems,
  {B}ratteli diagrams and dimension groups}, Ergodic Theory Dynam. Systems
  \textbf{19} (1999), no.~4, 953--993. \MR{1709427 (2000i:46062)}

\bibitem{F1}
S{\'e}bastien Ferenczi, \emph{Rank and symbolic complexity}, Ergodic Theory
  Dynam. Systems \textbf{16} (1996), no.~4, 663--682. \MR{1406427 (97g:58050)}

\bibitem{fisher}
Albert~M. Fisher, \emph{Nonstationary mixing and the unique ergodicity of adic
  transformations}, Stoch. Dyn. \textbf{9} (2009), no.~3, 335--391.
  \MR{2566907}

\bibitem{Fogg}
N.~Pytheas Fogg, \emph{Substitutions in dynamics, arithmetics and
  combinatorics}, Lecture Notes in Mathematics, vol. 1794, Springer-Verlag,
  Berlin, 2002, Edited by V. Berth{\'e}, S. Ferenczi, C. Mauduit and A. Siegel.
  \MR{1970385 (2004c:37005)}

\bibitem{glt}
Jean-Marc Gambaudo, Oscar Lanford, III, and Charles Tresser, \emph{Dynamique
  symbolique des rotations}, C. R. Acad. Sci. Paris S\'er. I Math. \textbf{299}
  (1984), no.~16, 823--826. \MR{772104}

\bibitem{HPS}
Richard~H. Herman, Ian~F. Putnam, and Christian~F. Skau, \emph{Ordered
  {B}ratteli diagrams, dimension groups and topological dynamics}, Internat. J.
  Math. \textbf{3} (1992), no.~6, 827--864. \MR{1194074 (94f:46096)}

\bibitem{HZ2}
C.~Holton and L.~Q. Zamboni, \emph{Directed graphs and substitutions}, Theory
  Comput. Syst. \textbf{34} (2001), no.~6, 545--564. \MR{1865811 (2002h:68174)}

\bibitem{HZ1}
Charles Holton and Luca~Q. Zamboni, \emph{Geometric realizations of
  substitutions}, Bull. Soc. Math. France \textbf{126} (1998), no.~2, 149--179.
  \MR{1675970 (2000c:37008)}

\bibitem{Keane}
Michael Keane, \emph{Interval exchange transformations}, Math. Z. \textbf{141}
  (1975), 25--31. \MR{0357739 (50 \#10207)}

\bibitem{Q}
Martine Queff{\'e}lec, \emph{Substitution dynamical systems---spectral
  analysis}, Lecture Notes in Mathematics, vol. 1294, Springer-Verlag, Berlin,
  1987. \MR{924156 (89g:54094)}

\bibitem{R1}
G.~Rauzy, \emph{Nombres alg\'ebriques et substitutions}, Bull. Soc. Math.
  France \textbf{110} (1982), no.~2, 147--178. \MR{667748 (84h:10074)}

\bibitem{R3}
\bysame, \emph{Rotations sur les groupes, nombres alg\'ebriques, et
  substitutions}, S\'eminaire de {T}h\'eorie des {N}ombres, 1987--1988
  ({T}alence, 1987--1988), Univ. Bordeaux I, Talence, 1988, pp.~Exp.\ No.\ 21,
  12. \MR{993118 (90g:11017)}

\bibitem{R2}
\bysame, \emph{Sequences defined by iterated morphisms}, Sequences
  ({N}aples/{P}ositano, 1988), Springer, New York, 1990, pp.~275--286.
  \MR{1040317 (91d:28038)}

\bibitem{ST}
J.~Schmeling and S.~Troubetzkoy, \emph{Interval translation mappings},
  Dynamical systems ({L}uminy-{M}arseille, 1998), World Sci. Publ., River Edge,
  NJ, 2000, pp.~291--302. \MR{1796167}

\bibitem{S1}
A.~Siegel, \emph{Spectral theory and geometric representation of
  substitutions}, Substitutions in dynamics, arithmetics and combinatorics,
  Lecture Notes in Math., vol. 1794, Springer, Berlin, 2002, pp.~199--252.
  \MR{1970392}

\bibitem{vee}
Peter Veerman, \emph{Symbolic dynamics of order-preserving orbits}, Phys. D
  \textbf{29} (1987), no.~1-2, 191--201. \MR{923891}

\bibitem{VL}
A.~M. Vershik and A.~N. Livshits, \emph{Adic models of ergodic transformations,
  spectral theory, substitutions, and related topics}, Representation theory
  and dynamical systems, Adv. Soviet Math., vol.~9, Amer. Math. Soc.,
  Providence, RI, 1992, pp.~185--204. \MR{1166202 (93i:46131)}

\bibitem{Sadicthesis}
Krzysztof Wargan, \emph{S-adic dynamical systems and {B}ratteli diagrams},
  ProQuest LLC, Ann Arbor, MI, 2002, Thesis (Ph.D.)--The George Washington
  University. \MR{2702703}

\end{thebibliography}

\end{document}